\documentclass[12pt]{article}
\usepackage{amssymb, amsthm}
\usepackage{amsmath}

\newtheorem{Rem}{Remark}
\newtheorem{sats}{Theorem}
\newtheorem{prop}{Proposition}
\newtheorem{lem}{Lemma}
\newtheorem{kor}{Corollary}
\newcommand{\banm}{\begin{anm}}
\newcommand{\eanm}{\end{anm}}

\begin{document}

\title
{The Dirichlet problem \\
for  non-divergence parabolic equations\\
with discontinuous in time coefficients. \\
}
\author
{Vladimir Kozlov, Alexander Nazarov}


\date{}


\maketitle

\centerline{Dedicated to V.A. Solonnikov on the occasion of his 75th
jubilee}

\section{Introduction}

In 2001  N.Krylov observed in \cite{Kr2} and \cite{Kr} that for
non-divergence parabolic equations  coercive estimates for
solutions can be proved even when the leading coefficients are
only measurable functions with respect to $t$. Namely, he
considered the equation
\begin{equation}\label{Jan1}
({\cal L}_0u)(x,t)\equiv\partial_tu(x,t)-
a^{ij}(t)D_iD_ju(x,t)=f(x,t)
\end{equation}
in $\mathbb R^n\times\mathbb R$, where $D_j=\partial/\partial x_j$
and $a^{ij}$ are measurable real valued functions of $t$
satisfying $a^{ij}=a^{ji}$ and
\begin{equation}\label{Jan2}
\nu|\xi|^2\le a^{ij}\xi_i\xi_j\le \nu^{-1}|\xi|^2, \qquad
\xi\in{\mathbb R}^n, \quad \nu=const>0.
\end{equation}
He proved that for $f\in L_{p,q}(\mathbb R^{n}\times\mathbb R)$
with $1<p,q<\infty$, where  $L_{p,q}(\Omega\times\mathbb R)$ is
the space of functions on $\Omega\times\mathbb R$  with finite
norm
\begin{equation}\label{Jan3}
\|f\|_{p,q}=\Big (\int\limits_{\mathbb R}\Big
(\int\limits_{\Omega} |f(x,t)|^pdx\Big )^{q/p}dt\Big )^{1/q},
\end{equation}
equation (\ref{Jan1}) has a unique solution such that
$\partial_tu$ and $D_iD_ju$ belong to $L_{p,q}(\mathbb
R^n\times\mathbb R)$ and
\begin{equation}\label{Jan4}
\|\partial_tu\|_{p,q}+\sum_{ij} \|D_iD_ju\|_{p,q}\leq
C\|f\|_{p,q}\;.
\end{equation}

Let us turn to  the Dirichlet boundary value problem in the
half-space $\mathbb R_+^{n}=\{x=(x',x_n)\in\mathbb R^n:x_n>0\}$.
Now equation (\ref{Jan1}) is satisfied for $x_n>0$ and $u=0$ for
$x_n=0$. The following weighted coercive  estimate
\begin{equation}\label{Jan5}
\|x_n^\mu\partial_tu\|_{p,q}+\sum_{ij} \|x_n^\mu
D_iD_ju\|_{p,q}\leq C\|x_n^\mu f\|_{p,q}\;,
\end{equation}
was proved in \cite{Kr2}, where $1<p,q<\infty$ and
$\mu\in\,(1-1/p,2-1/p)$. Furthermore from \cite{DoKr} and
\cite{Kim}, it follows that the solution of the Dirichlet problem
to (\ref{Jan1}) satisfies estimate (\ref{Jan4}) for $\mu=0$ and $p=q$,
$p\in\,(1,\infty)$.

One of the main results of this paper is the proof of  estimate
(\ref{Jan5}) for solutions of the Dirichlet problem to (\ref{Jan1})
 for arbitrary $p$ and $q$ from $(1,\infty)$ and for $\mu$
satisfying
\begin{equation}\label{Mart2a}
-1/p<\mu<2-1/p\,.
\end{equation}
We also prove analogs of estimates (\ref{Jan4}) and (\ref{Jan5}),
where the norm $\|\cdot \|_{p,q}$ is replaced by
$$
|\!|\!|f|\!|\!|_{p,q}=\Big (\int\limits_\Omega\Big
(\int\limits_{\mathbb R} |f(x,t)|^qdt\Big )^{p/q}dx\Big )^{1/p}.
$$
These norms and corresponding spaces, which will be denoted by
$\widetilde L_{p,q}(\Omega\times \mathbb R)$, play important role
in the theory of quasilinear  non-divergence parabolic equation
(see \cite{Na}).

In  Sect. \ref{solv} we give some applications of our results to
the Dirichlet problem for linear and quasi-linear non-divergence
parabolic equations with discontinuous in time coefficients in
cylinders $\Omega\times (0,T)$, where $\Omega$ is a bounded domain in $\mathbb
 R^n$. We prove solvability results in weighted $L_{p,q}$ and $\widetilde
 L_{p,q}$ spaces, where the weight is a power of the distance to the boundary of $\Omega$.
The smoothness of the boundary is characterized by smoothness of
local isomorphisms in neighborhoods of boundary points, which
flatten the boundary. In particular, if the boundary is of the class
${\cal C}^{1,\delta}$ with $\delta\in [0,1]$, then for solutions to the
linear problem (\ref{Jan1}) in $\Omega\times (0,T)$, where the
coefficients $a^{ij}$ may depend on $x$ (namely, $a^{ij}\in
C(\Omega\to L^\infty (0,T))$), with zero initial and Dirichlet
boundary conditions the following coercive estimate is proved in
Theorem \ref{linear}:
\begin{eqnarray*}
&&\|(\widehat d(x))^\mu\partial_tu\|_{p,q}+\sum_{ij} \|(\widehat
d(x))^\mu D_iD_ju\|_{p,q}\leq
C\|(\widehat d(x))^\mu f\|_{p,q},\\
&&|\!|\!|(\widehat d(x))^\mu\partial_tu|\!|\!|_{p,q}+\sum_{ij}
|\!|\!|(\widehat d(x))^\mu D_iD_ju|\!|\!|_{p,q}\leq
C|\!|\!|(\widehat d(x))^\mu f|\!|\!|_{p,q},
\end{eqnarray*}
where $\mu$, $p$, $q$ and $\delta$ satisfy $1<p,q<\infty$, \
$1-\delta-\frac{1}{p}<\mu<2-\frac{1}{p}$. Here we use the notation
${\cal C}^{1,0}$ for boundaries of the class ${\cal C}^1$. For $p=q$ and
$\delta=0$ this estimated was proved in \cite{KyKr}.

In order to prove estimate (\ref{Jan5}) we use an approach based
on the study  of the Green function. We obtain point-wise
estimates for the Green function for the Dirichlet problem in the
half-space and its derivatives, see Sect.\ref{Rn+1}. The main
ingredient in the proof is the decomposition of the kernel
$$ \frac {x_n^{\mu}}{y_n^{\mu}}\, D_{x_i}D_{x_j}\Gamma^{\cal D}(x,y;t, s)$$
into the sum of truncated singular kernel $\chi_{\{x_n>\sqrt{t-
s}\}}D_{x_i}D_{x_j}\Gamma(x,y;t,s)$ and the complement kernel, see
Sect.\ref{Rn+2}. Here $\Gamma$ and $\Gamma^{\cal D}$ are the Green
functions for the whole space and for the half-space respectively.
The boundedness of singular operators with truncated kernels in
$L_{p,q}$ and $\widetilde L_{p,q}$ spaces is proved in
Sect.\ref{Rn}. Then, using local estimates for solutions to
parabolic equations in the half-space, we show that the complement
kernels have weak singularities and give estimates of the norms of
corresponding operators in $L_{p,q}$ and $\widetilde L_{p,q}$
spaces. This leads to the proof of (\ref{Jan5}) under condition
(\ref{Mart2a}). Similar decompositions of the Green function  were
used by V.A. Solonnikov in \cite{Sol1} and \cite{Sol}.

\medskip

We shall use the following notation: $x=(x_1,\dots,x_n)=(x',x_n)$
is a point in $\mathbb R^n$; ${\mathbb R}^n_+=\{x \in{\mathbb
R}^n: x_n>0\}$ is a half-space;
$$
Q_R(x^0,t^0)=\{ (x,t): |x-x^0|<R,\; 0<t^0-t<R^2\}
$$
is a cylinder;
$$
Q_R^+(x^0,t^0)=\{ (x,t): |x-x^0|<R,\; x_1>0,\; 0<t^0-t<R^2\}.
$$
The last notation will be used only for $x^0\in\overline{\mathbb
R^n_+}$. We adopt the convention regarding summation from $1$ to
$n$ with respect to repeated indices.  We use the letter $C$ to
denote various positive constants. To indicate that $C$ depends on
some parameter $a$, we sometimes write $C_a$.

\section{The estimates in the whole space}\label{Rn}

Let us consider  equation (\ref{Jan1}). Using the Fourier
transform with respect to  $x$ one can obtain the following
representation of solution through the right-hand side:
\begin{equation}\label{TTN1}
u(x,t)=\int\limits_{-\infty}^t\int\limits_{{\mathbb R}^n}
\Gamma(x,y;t,s) f(y,s)\ dy\,ds,
\end{equation}
where $\Gamma$ is the Green function of the operator ${\cal L}_0$
 given by
\begin{equation*}
\Gamma(x,y;t,s)= \frac { \det\big(\int_s^t
A(\tau)d\tau\big)^{-\frac 12}} {(4\pi)^{\frac n2}}    \exp
\bigg(-\frac {\Big(\big(\int_s^t A(\tau) d\tau\big)^{-1}
(x-y),(x-y)\Big)}4\bigg)
\end{equation*}
for $t>s$ and $0$ otherwise.  Here  $A(t)$ is  the matrix $\{
a^{ij}(t)\}_{i,j=1}^n$. The above representation implies, in
particular,
 the following estimates for $\Gamma$.

\begin{prop} Let $\alpha$ and $\beta$ be two arbitrary multi-indices. Then
\begin{equation}\label{May0}
|D_x^\alpha D_y^\beta \Gamma(x,y;t,s)|\le C\,(t-s)^{-\frac
{n+|\alpha|+|\beta|}2}
\,\exp\left(-\frac{\sigma|x-y|^2}{t-s}\right)
\end{equation}
and
\begin{equation}\label{May0'}
|\partial_tD_x^\alpha D_y^\beta \Gamma(x,y;t,s)|\le
C\,(t-s)^{-\frac {n+|\alpha|+|\beta|}2-1}
\,\exp\left(-\frac{\sigma|x-y|^2}{t-s}\right),
\end{equation}
for $x,y\in\mathbb R^n$ and $s<t$. Here $\sigma$ depends only on
the ellipticity constant $\nu$ and $C$ may depend on $\nu$,
$\alpha$ and $\beta$.
\end{prop}

In what follows we denote by the same letter the kernel  and the
corresponding integral operator, i.e.
\begin{equation}\label{Feb15a}
({\cal K}h)(x,t)=\int\limits_{-\infty}^t\int\limits_{\mathbb R^n}
{\cal K}(x,y;t,s)h(y,s)\,dyds.
 \end{equation}

In order to prove an analog of estimate (\ref{Jan4}) for
$\widetilde L_{p,q}$ we need the following lemma. We introduce the
kernels ${\mathfrak
G}_{ij}(x,y;t,s)=D_{x_i}D_{x_j}\Gamma(x,y;t,s)$. Thus the notation
${\mathfrak G}_{ij}$ is used both for the kernel and for the
corresponding operator defined by (\ref{Feb15a}).

\begin{lem}\label{weak_1}
Let a function $h$ be supported in the cylinder $|y-y^0|\le\delta$
and satisfy $\int h(y,s)\ dy\equiv 0$ for almost all $s$. Then
\begin{equation}\label{Feb10a}
\int\limits_{|x-y^0|>2\delta}\Vert ({\mathfrak G}_{ij}
h)(x,\cdot)\Vert_q\ dx\le C\,|\!|\!| h|\!|\!|_{1,q},
\end{equation}

\noindent where $C$ does not depend on $\delta$ and $y^0$.
\end{lem}

\begin{proof}
Due to $\int h(y,s)\ dy\equiv 0$, we have
$$({\mathfrak G}_{ij}h)(x,t)=\int\limits_{-\infty}^{t}\int\limits_{\mathbb R^n}
\Bigl({\mathfrak G}_{ij}(x,y;t,s)-{\mathfrak
G}_{ij}(x,y^0;t,s)\Bigr)\, h(y,s)\ dy\ ds.
$$
Using estimate (\ref{May0}) for $\nabla_y{\mathfrak
G}_{ij}(x,y;t,s)$, we obtain
$$
\Bigl|{\mathfrak G}_{ij}(x,y;t,s)-{\mathfrak
G}_{ij}(x,y^0;t,s)\Bigr|\le \frac {C\delta}{(t-s)^{\frac
{n+3}2}}\,\exp \left(-\frac {\sigma_1|x-y|^2}{t-s} \right)
$$
for $|y-y^0|\le\delta$ and $|x-y^0|\ge2\delta$, where $\sigma_1$
is a positive constant depending on $\sigma$. Applying this
estimate together with the H\"older inequality, we get
\begin{eqnarray}\label{est_gamma}
|({\mathfrak G}_{ij}h)(x,t)|&\le &C\delta\int\limits_{\mathbb R^n}
\left(\ \int\limits_{-\infty}^{t} \exp \left(-\frac
{\sigma_1|x-y|^2}{t-s}\right) \frac {|h(y,s)|^q\ ds}{(t-s)^{\frac
{n+3}2}}\right)^{\frac 1q}\nonumber\\
&\times&\left(\ \int\limits_{-\infty}^{t} \exp \left(-\frac
{\sigma_1|x-y|^2}{t-s} \right) \frac {ds}{(t-s)^{\frac
{n+3}2}}\right)^{\frac 1{q'}}dy.
\end{eqnarray}

 \noindent Using the change of variable $\tau=(t-s)|x-y|^{-2}$ in
the last integral over $(-\infty,t)$, we estimate it by
$C|x-y|^{-(n+1)}$. Therefore,
\begin{equation*}
|({\mathfrak G}_{ij}h)(x,t)|\le C\int\limits_{\mathbb R^n} \left(\
\int\limits_{-\infty}^{t} \exp \left(-\frac
{\sigma_1|x-y|^2}{t-s}\right) \frac {|h(y,s)|^q\ ds}{(t-s)^{\frac
{n+3}2}}\right)^{\frac 1q}\!\frac{\delta dy}{|x-y|^{(n+1)/q'}}
\end{equation*}
for $|x-y_0|>2\delta$.
 Integrating this estimate and applying   Minkowski's inequality, we
obtain
\begin{eqnarray*}
\int\limits_{|x-y^0|>2\delta}\! \Vert ({\mathfrak
G}h)(x,\cdot)\Vert_q\ dx&\leq&
C\!\!\int\limits_{|x-y^0|>2\delta}\int\limits_
{|y-y^0|<\delta}\ \frac {\delta\ dy\ dx} {|x- y|^{\frac {n+1}{q'}}}\\
&\times &\bigg(\ \int\limits_{-\infty}^{\infty}
\int\limits_{-\infty}^t\exp \bigg(-\frac {\sigma_1|x-y|^2}{t-s}
\bigg) \frac {|h(y,s)|^q\ ds\ dt}
{({t-s})^{\frac {n+3}2}}\bigg)^{\frac 1q} \\
&\le& C\int\limits_{|y-y^0|<\delta}\| h(y,\cdot)\|_q\ dy
\sup\limits_{|y-y^0|<\delta} \int\limits_ {|x-y^0|>2\delta}\frac
{\delta\ dx}
{|x- y|^{\frac {n+1}{q'}}}\\
&\times&\left(\ \sup\limits_{s\ge 0}\int\limits_s^{\infty} \exp
\left(-\frac {\sigma_1|x-y|^2}{t- s} \right) \frac {dt}
{({t-s})^{\frac {n+3}2}}\right)^{\frac 1q}.
\end{eqnarray*}
Using again the change of variable $\tau=(t-s)|x-y|^{-2}$ in the
last integral, we estimate it by $C|x-y|^{-(n+1)}$, and hence
$$\int\limits_{|x-y^0|>2\delta}\Vert ({\mathfrak G}_{ij}h)(x,\cdot)\Vert_q\
dt\le C\,|\!|\!| h|\!|\!|_{1,q}\, \sup\limits_{|y-y^0|<\delta}
\int\limits_ {|x-y^0|>2\delta}\frac {\delta\ dx} {|x- y|^{n+1}}\le
C\,|\!|\!| h|\!|\!|_{1,q},
$$
which coincides with (\ref{Feb10a}).
\end{proof}

\begin{sats}\label{space}
Let $p, q\in\,(1,\infty)$ and $f\in \widetilde L_{p,q}(\mathbb
R^n\times\mathbb R)$.
 Then the solution of equation {\rm (\ref{Jan1})} given by {\rm (\ref{TTN1})} satisfies
\begin{equation}\label{est_space}
|\!|\!|\partial_t u|\!|\!|_{p,q}+\sum_{ij}|\!|\!|
D_iD_ju|\!|\!|_{p,q}\le C\ |\!|\!| f|\!|\!|_{p,q},
\end{equation}
where $C$ depends only on $\nu$, $p$, $q$.
\end{sats}

\begin{proof} From (\ref{Jan4}) it follows boundedness of
${\mathfrak G}_{ij}$ in $L_q(\mathbb R^n\times \mathbb R)$,
$1<q<\infty$, which implies the first condition in \cite[Theorem
3.8]{BIN} with $p=r=q$. Lemma \ref{weak_1} is equivalent to the
second condition in this theorem with $p=q$. Therefore, we can
apply Theorem 3.8 \cite{BIN} to the operator ${\mathfrak G}_{ij}$
and it ensures that this operator is bounded in $\widetilde
L_{p,q}(\mathbb R^n\times \mathbb R)$ for any $p\in\,(1,q)$. For
$p>q$ its boundedness follows from the boundedness of the adjoint
operator in $\widetilde L_{p',q'}(\mathbb R^n\times\mathbb R)$
which is proved by verbatim repetition of previous arguments.

Thus, we obtain the estimate of the second term in
(\ref{est_space}). The estimate of the first term follows now from
(\ref{Jan1}).
\end{proof}

In Section \ref{Rn+2} we need the following estimate for the
operator corresponding to the  truncated kernels
$$
\widehat{\mathfrak G}_{ij}(x,y;t,s)= \chi_{\{x_n>\sqrt{t- s}\}}\,
D_{x_i}D_{x_j}\Gamma(x,y;t,s),
$$
where $\chi$ stands for the indicator function.
\begin{sats}\label{trunc}
Let $p, q\in\,(1,\infty)$, and let $j\ne n$. Then the integral
operator $\widehat{\mathfrak G}_{ij}$ is bounded both in
$L_{p,q}(\mathbb R^n\times \mathbb R)$ and $\widetilde
L_{p,q}(\mathbb R^n\times \mathbb R)$ spaces.
\end{sats}

\begin{proof} Step 1. {\em Boundedness in $L_{2}(\mathbb R^n\times \mathbb R)$}.
Since the boundedness of the operators ${\mathfrak G}_{ij}$ and
$\chi_{\{x_n>0\}}\, {\mathfrak G}_{ij}$ in $L_2$ follows from
(\ref{Jan4}), it suffices to show that the operator with kernel
$$
\widetilde{\mathfrak G}_{ij}(x,y;t,s)=\chi_{\{x_n\in\,(0,\sqrt{t-s})\}}\,D_{x_i}D_{x_j}\Gamma(x,y;t,s)
$$
is bounded.

The Fourier transform with respect to $x$ gives
\begin{eqnarray*}
{\cal F}(\widetilde{\mathfrak
G}_{ij}h)(\xi',\eta_n,t)&=&\int\limits_{-\infty}^t
\int\limits_{\mathbb R}\exp \Big(-\Big(\big( {\textstyle\int_s^t
A(\tau)
d\tau} \big) \xi,\xi\Big)\Big)\,\xi_i\xi_j\,({\cal F}h)(\xi,s)\\
&\times&\frac
{\exp(i\sqrt{t-s}(\eta_n-\xi_n))-1}{\eta_n-\xi_n}\,d\xi_nds.
\end{eqnarray*}
Using the assumption $j\ne n$ and (\ref{Jan2}) we obtain
\begin{eqnarray*}
|{\cal F}(\widetilde{\mathfrak G}_{ij}h)(\xi',\eta_n,t)|&\le&
\int\limits_{0}^\infty\int\limits_{\mathbb R}\exp \big(-\nu |\xi|^2
s\big)\,
|\xi|\,|\xi'|\\
&\times& |({\cal F}h)(\xi,t-s)|\ \phi\big(\sqrt s
(\eta_n-\xi_n)\big)\sqrt{s} \,d\xi_nds,
\end{eqnarray*}
where $\phi(\tau)=\big|\tau^{-1}(\exp(i\tau)-1)\big|$. By the
H\"older inequality,
\begin{eqnarray*}
&&|{\cal F}(\widetilde{\mathfrak G}_{ij}h)(\xi',\eta_n,t)|^2\le
\int\limits_{\mathbb R} |\xi'|^2 \int\limits_0^\infty \exp
\big(-\nu |\xi|^2 s\big)\sqrt{s}\,ds\,d\xi_n\\
&&\times\int\limits_{\mathbb R}\int\limits_0^\infty |({\cal
F}h)(\xi,t-s)|^2 |\xi|^2\exp \big(-\nu|\xi|^2 s\big) \phi^2(\sqrt s
(\eta_n-\xi_n)) \sqrt{s}\,dsd\xi_n.
\end{eqnarray*}
First two integrals give a constant.  Therefore,
\begin{eqnarray*}
&&\int\limits_{\mathbb R}\int\limits_{\mathbb R} |{\cal
F}(\widetilde{\mathfrak G}_{ij}h)(\xi',\eta_n,t)|^2d\eta_ndt\le
C\int\limits_{\mathbb R}\phi^2(\sqrt s (\eta_n- \xi_n))
\sqrt{s}\,d\eta_n\times\\
&&\!\!\!\int\limits_{\mathbb R}\!\!\! \int\limits_{\mathbb R}|({\cal
F}h)(\xi,\tau)|^2 d\tau \int\limits_0^\infty|\xi|^2\exp \big(-\nu
|\xi|^2 s\big) dsd\xi_n \le\!\! C\!\!\int\limits_{\mathbb R}
\!\!\!\int\limits_{\mathbb R}|({\cal F}h)(\xi,\tau)|^2 d\tau d\xi_n.
\end{eqnarray*}
We integrate this inequality with respect to $\xi'$, and the
statement follows by the Plancherel theorem.\medskip

Step 2. {\em Boundedness in $L_p(\mathbb R^n\times \mathbb R)$}.
For a function $h$ supported in the cylinder $Q_\delta(y^0,s^0)$
and satisfying $\int h(y,s)\ dyds= 0$ the following inequality is
valid:
\begin{equation}\label{Jan23cq}
\int\limits_{\mathbb R^n\backslash Q_{2\delta}(y^0,s^0)}
|(\widehat {\mathfrak G}_{ij} h)(x,t)|\ dxdt\le C\,\| h\|_{1},
\end{equation}
where $C$ does not depend on $\delta$, $y^0$ and $s^0$. Since the
proof of this inequality repeats, with some simplifications, the
proof of estimate (\ref{Jan23c}) below, we confine ourselves  to
proving (\ref{Jan23c}). By \cite[Theorem 3 and \S5.3]{St}, the
estimate (\ref{Jan23cq}) and Step 1 provide the boundedness of
$\widehat{\mathfrak G}_{ij}$ in $L_p(\mathbb R^n\times \mathbb R)$
for $1<p<2$. As in the proof of Theorem \ref{space}, the
boundedness for $2<p<\infty$ is proved  by duality argument.\medskip

Step 3. {\em Boundedness in $L_{p,q}(\mathbb R^n\times \mathbb
R)$}. Next, let us show that for a function $h$ supported in the
layer $|s-s^0|\le\delta$ and satisfying $\int h(y,s)\ ds\equiv 0$
for almost all $y$,
\begin{equation}\label{Jan23c}
\int\limits_{|t-s^0|>2\delta}\Vert (\widehat {\mathfrak G}_{ij}
h)(\cdot,t)\Vert_p\ dt\le C\,\| h\|_{p,1},
\end{equation}

\noindent where $C$ does not depend on $\delta$ and $s^0$. Since
$(\widehat {\mathfrak G}_{ij} h)(x,t)=0$ for $s>t+2\delta$, the
integral in (\ref{Jan23c}) is actually taken over $t>s_0+2\delta$.
By $\int h(y, s)\ ds\equiv 0$, we have
\begin{equation}\label{difference1}
(\widehat{\mathfrak G}_{ij}h)(x,t)=
\int\limits_{-\infty}^{t}\int\limits_{\mathbb R^n_+}
\Bigl(\widehat{\mathfrak G}_{ij}(x,y,t, s)- \widehat{\mathfrak
G}_{ij}(x,y,t, s^0)\Bigr)\, h(y, s)\ dy ds.
\end{equation}
For $| s- s^0|<\delta$ and $t- s^0>2\delta$, inequalities
(\ref{May0}) and (\ref{May0'})  imply
\begin{eqnarray*}
&&\left|\widehat{\mathfrak G}_{ij}(x,y,t, s)- \widehat{\mathfrak
G}_{ij}(x,y,t, s^0)\right| \le\int\limits_{s^0}^s|\partial_\tau
{\mathfrak G}_{ij}(x,y,t, \tau)|\,d\tau\\
&&+|{\mathfrak G}_{ij}(x,y,t,t-x_n^2)|\,\chi_{x_n^2\in[t-s,t-s^0]}
\le \frac {C\delta} {(t- s)^{\frac {n+4}2}}\, \exp
\left(-\frac {\sigma|x-y|^2}{t- s}\right)\\
&&+\frac {C\chi_{x_n^2\in[t-s,t-s^0]}}{(t-s)^{\frac {n+2}2}}
\,\exp\left(-\frac {\sigma|x-y|^2}{t-s}\right) =: {\cal
I}_1(x,y,t, s)+{\cal I}_2(x,y,t, s).
\end{eqnarray*}
Using this in estimating of the right-hand side in
(\ref{difference1}), we obtain
$$
\int\limits_{|t- s^0|>2\delta}\Vert (\widehat{\mathfrak
G}_{ij}h)(\cdot, t)\Vert_p\ dt \le \int\limits_{|t-
s^0|>2\delta}\Vert ({\cal I}_1h)(\cdot, t)\Vert_p\ dt
+\int\limits_{|t- s^0|>2\delta}\Vert ({\cal I}_2h)(\cdot,
t)\Vert_p\ dt.
$$
The first term is estimated by $\Vert h\Vert_{p,1}$ in the same
way as (\ref{est_gamma}). Let us estimate the second term. We have
\begin{eqnarray*}
|({\cal I}_2h)(x,t)|&\le& C\,\int\limits_{-\infty}^{t}\frac
{\chi_{\{x_n^2\in(t-s,t-s^0)\}}\
ds}{t-s}\biggl(\!\int\limits_{\mathbb R^n_+}\exp
\left(-\frac{\sigma|x-y|^2}{t-s} \right) \frac {dy}{(t- s)^{\frac
n2}} \biggr)^{\frac{1}{p'}}\\
&\times&\biggl(\!\int\limits_{\mathbb R^n_+} \exp
\left(-\frac{\sigma|x-y|^2} {t- s}\right)\frac {|h(y, s)|^p\ dy}
{(t- s)^{\frac n2}}\biggr)^{\frac{1}{p}}\\
\end{eqnarray*}
The last integral is bounded uniformly with respect to  $x$, $t$
and $s$. Since $| s- s^0|<\delta$, we have
$[t-s,t-s^0]\subset[t-s^0-\delta,t-s^0+\delta]$. Using the
Minkowski inequality, we obtain
\begin{eqnarray*}
\int\limits_{|t- s^0|>2\delta}\Vert ({\cal I}_2h)(\cdot, t)\Vert_p\
dt&\le &C\,\int\limits_{ s^0-\delta}^{ s^0+\delta} \|
h(\cdot,s)\|_p\ ds \,\int\limits_
{ s^0+2\delta}^{\infty}\frac {dt}{t-s}\\
&\times &\biggl(\ \sup\limits_y\int\limits_{\mathbb R^{n-1}}
\int\limits_{\sqrt {t-s^0-\delta}}^{\sqrt {t-s^0+\delta}} \exp
\left(-\frac {\sigma|x-y|^2}{t-s} \right) \frac {dx'dx_n}
{(t-s)^{\frac n2}}\biggr)^{\frac{1}{p}}.
\end{eqnarray*}
Denote by $I_2$ the integral in the last large brackets. Using the
change of variables $x=z\sqrt{t- s}$, $y=w\sqrt{t- s}$ and
integrating w.r.t. $z'$, we obtain
\begin{eqnarray*}
I_2&=&C\,\!\int\limits_{\sqrt {\frac {t-s^0-\delta}{t-s}}} ^{\sqrt
{\frac {t-s^0+\delta}{t-s}}} \exp\left(-\sigma |z_n-w_n|^2 \right)\ dz_n  \\
&\le &C\, \left(\mbox {\large $\textstyle \sqrt {\frac
{t-s^0+\delta}{t-s}}- \sqrt {\frac {t-s^0-\delta}{t-s}}$ }\right)\le
C\,\frac {\delta}{t-s}.
\end{eqnarray*}
Thus,
$$\int\limits_{|t- s^0|>2\delta}\Vert ({\cal I}_2h)(\cdot, t)\Vert_p\ dt\le
C\,\Vert h\Vert_{p,1}\, \sup\limits_{| s- s^0|<\delta}
\int\limits_ { s^0+2\delta}^{\infty}\frac {\delta^{1/p} dt} {(t-
s)^{1+1/p}}\le C\,\Vert h\Vert_{p,1}.
$$
By \cite[Theorem 3.8]{BIN}, the estimate (\ref{Jan23c}) and Step 2
provide the boundedness of $\widehat{\mathfrak G}_{ij}$ in
$L_{p,q}(\mathbb R^n\times \mathbb R)$ for $q\in\,(1,p)$. Using
duality argument, we obtain boundedness for
$q\in\,(p,\infty)$.\medskip

Step 4. {\em Boundedness  in $\widetilde L_{p,q}(\mathbb R^n\times
\mathbb R)$}. For a function $h$ supported in the cylinder
$|y-y^0|\le\delta$ and satisfying $\int h(y,s)\ dy\equiv 0$, the
following inequality
\begin{equation}\label{Jan23cqq}
\int\limits_{|x-y^0|>2\delta}\Vert (\widehat {\mathfrak G}_{ij}
h)(x,\cdot)\Vert_q\ dx\le C\,|\!|\!| h|\!|\!|_{1,q},
\end{equation}
holds, where $C$ does not depend on $\delta$ and $y^0$. The proof
of (\ref{Jan23cqq}) repeats literally the proof of Lemma
\ref{weak_1}. By \cite[Theorem 3.8]{BIN}, the estimate
(\ref{Jan23cqq}) and Step 2 provide   boundedness of
$\widehat{\mathfrak G}_{ij}$ in $\widetilde L_{p,q}(\mathbb
R^n\times \mathbb R)$ for $p\in\,(1,q)$. The boundedness  for
$p\in\,(q,\infty)$ follows by duality argument.
\end{proof}

\section{The Green function in a half-space}\label{Rn+1}

We denote  by $\Gamma^{\cal D}(x,y,t,s)$ the Green function of the
operator ${\cal L}_0$ in the half-space ${\mathbb R}^n_+$ subject
to the homogeneous Dirichlet boundary condition on the boundary
$x_n=0$. From the maximum principle it follows that
$0\le\Gamma^{\cal D}(x,y,t,s)\le \Gamma(x,y,t,s)$ and hence by
(\ref{May0})
\begin{equation}\label{Okt1}
|\Gamma^{\cal D}(x,y,t,s)|\le C\,(t-s)^{-\frac{n}{2}}\exp
\left(-\frac{\sigma|x-y|^2}{t-s}\right).
\end{equation}
The aim of this section is to prove point-wise estimates for
derivatives of $\Gamma^{\cal D}$.

We need a well-known local gradient estimate for solutions to
parabolic equations in a half-space. The next statement can be
found (up to scaling) in Ch. III, Sect. 11 and 12 in \cite{LSU}.

\begin{prop}\label{Pr1a} {\rm (i)} Let $u\in W^{2,1}_2(Q_R(x^0,t^0))$ solve the equation ${\cal
L}_0u=0$  in $Q_R(x^0,t^0)$ . Then
$$
|D u|\le \frac{C}{R}\sup_{Q_{R}(x^0,t^0)} u\qquad \mbox{in}\quad
Q_{R/2}(x^0,t^0).
$$
{\rm (ii)} Let $u\in W^{2,1}_2(Q_R^+(x^0,t^0))$ solve the equation
${\cal L}_0u=0$  in $Q_R^+(x^0,t^0)$ and let $u|_{x_n=0}=0$. Then
$$
|D u|\le \frac{C}{R}\sup_{Q_{R}^+(x^0,t^0)} u\qquad \mbox{in}\quad
Q_{R/2}^+(x^0,t^0).
$$
Here $C$ depends only on $\nu$.

\end{prop}

Iterating the above inequalities we arrive at

\begin{lem}\label{L1a}
{\rm (i)} Let $u\in W^{2,1}_2(Q_R(x^0,t^0))$ solve the equation
${\cal L}_0u=0$  in $Q_R(x^0,t^0)$ . Then
$$
|D^\alpha u|\le \frac{C}{R^{|\alpha|}}\sup_{Q_{R}(x^0,t^0)}
u\qquad \mbox{in}\quad Q_{R/2^{|\alpha|}}(x^0,t^0).
$$

{\rm (ii)} Let $u\in W^{2,1}_2(Q_R^+(x^0,t^0))$ solve the equation
${\cal L}_0u=0$  in $Q_R^+(x^0,t^0)$ and let $u|_{x_n=0}=0$. If
$\alpha_1\le 1$ then
$$
|D^\alpha u|\le \frac{C}{R^{|\alpha|}}\sup_{Q_{R}^+(x^0,t^0)}
u\qquad \mbox{in}\quad Q_{R/2^{|\alpha|}}^+(x^0,t^0).
$$
Here $C$ depends only on $\nu$ and $\alpha$.

\end{lem}

In the next lemma we give local estimates of the normal
derivatives.
\begin{lem}\label{L1b}

Let $u\in W^{2,1}_2(Q_R^+(x^0,t^0))$
 solve the equation ${\cal L}_0u=0$  in $Q_R^+(x^0,t^0)$ and let
$u|_{x_n=0}=0$. Then for $k\ge 2$ and arbitrary small
$\varepsilon>0$ the following inequality
\begin{equation}\label{Kop1a}
x_n^{k-2+\varepsilon}|D_{x_n}^k u|\le
\frac{C}{R^{2-\varepsilon}}\sup_{Q_{R}^+(x^0,t^0)} u\qquad
\mbox{in}\quad Q_{R/8^{|\alpha|}}^+(x^0,t^0)
\end{equation}
holds, where positive constant $C$ depends on $\nu$, $k$ and
$\varepsilon$.
\end{lem}

\begin{proof}
 If $R/4\le x^0_n$ then
(\ref{Kop1a}) follows from Lemma \ref{L1a}(i). Suppose $R/4>
x^0_n$. Let us prove first that for every $\alpha\in\,]0,1[$
\begin{equation}\label{Sept3}
\sup_{Q_{R/4}^+(x^0,t^0)}\frac{
|Du(x,t)-Du(y,s)|}{|x-y|^{\alpha}+|t-s|^{\frac\alpha2}}\le
CR^{-1-\alpha}\sup_{Q_{R}^+(x^0,t^0)} |u|.
\end{equation}
Let $\eta=\eta(x,t)$  be a smooth function which is equal to $1$
for $|t|\le 1/16$, $|x|\le 1/4$ and equal to $0$ for $t\ge 1/4$,
$|x|\ge 1/2$. We put $\eta_R(x,t)=\eta((t-t^0)/R^2,(x-x^0)/R)$. We
write the equation ${\cal L}_0u=0$ as
\begin{eqnarray}\label{Sept1}
\partial_t(\eta_Ru)&-&a^{nn}\Delta(\eta_Ru)\nonumber\\
&=&\eta_R\widetilde a^{ij}D_iD_ju+
(\partial_t\eta_R)u-a^{nn}(u\,\Delta\eta_R+2D_j\eta_R\,D_ju),
\end{eqnarray}
where $\widetilde a^{ij}(t)=a^{ij}(t)-a^{nn}(t)\delta^{ij}$. We
note that for the operator $\partial_t-a^{nn}\Delta$ with zero
Dirichlet boundary condition estimate (\ref{est_space}) is also
valid since by using the odd extension of solution and the
right-hand side we can reduce the Dirichlet problem in the
half-space to the problem for odd functions in the whole space.
Therefore, applying estimate (\ref{est_space}) with $q=p$ to
equation (\ref{Sept1}) we obtain
\begin{eqnarray*}
\|\partial_tu\|_{L^p(Q^+_{R/4}(x^0,t^0))}&+&
\|D^2u\|_{L^p(Q^+_{R/4}(x^0,t^0))}
\le C\Big(\sum_{j=1}^{n-1}\|D_jDu\|_{L^p(Q^+_{R/2}(x^0,t^0))}\\
&+&R^{-1}\|Du\|_{L^p(Q^+_{R/2}(x^0,t^0))}+
R^{-2}\|u\|_{L^p(Q^+_{R/2}(x^0,t^0))}\Big).
\end{eqnarray*}
Now using  Lemma \ref{L1a}(ii) for estimating the terms in the
right-hand side we arrive at
\begin{equation}\label{Sept2}
\|\partial_tu\|_{L^p(Q^+_{R/4}(x^0,t^0))}+\|D^2u\|_{L^p(Q_{R/4}(x^0,t^0))}
\le C\,R^{\frac{n+2-2p}p}\sup_{Q_{R}(x^0,t^0)} |u|.
\end{equation}
Next, we use the following Morrey-type inequality (see
\cite[Ch.2, Lemma 3.3]{LSU})
\begin{multline*}
\sup_{Q_{R/4}(x^0,t^0)}\frac{
|Du(x,t)-Du(y,s)|}{|x-y|^{\gamma}+|t-s|^{\frac\gamma2}}\le \\
\le CR^{1-\gamma -\frac {n+2}p}\Big
(\|D^2u\|_{L^p(Q^+_{R/4}(x^0,t^0))}+R^{-2}\|u\|_{L^p(Q^+_{R/4}(x^0,t^0))}\Big),
\end{multline*}
which is valid for $p>(n+2)/(1-\gamma)$. Estimating the right-hand
side here by (\ref{Sept2}), we obtain (\ref{Sept3}).

Now we are in position to complete the proof of inequality
(\ref{Kop1a}). We start with the estimate
\begin{equation*}
\rho^{-1+k-\gamma}\sup_{Q_{\rho_1}(x^1,t^1)} |D^k_{x_n} u|\le C
\sup_{Q_{2\rho}(x^1,t^1)} \frac {|Du(x,t)-Du(y,s)|}
{|x-y|^{\gamma}+|t-s|^{\frac\gamma2}},
\end{equation*}
(here $(x^1,t^1)\in Q^+_{R/8^{|\alpha|}}(x^0,t^0)$,
$\rho_1=\rho/2^{k-1}$ and $\rho=x^1_n/2$) which follows from Lemma
\ref{L1a}(i). Since $Q_{2\rho}(x^1,t^1)\subset
Q^+_{R/4}(x^0,t^0)$, we obtain from the last inequality that
\begin{equation*}
\rho^{-1+k-\gamma}\sup_{Q_{\rho_1}(x^1,t^1)} |D^k_{x_n} u|\le
C\sup_{Q^+_{R/4}(x^0,t^0)} \frac
{|Du(x,t)-Du(y,s)|}{|x-y|^{\gamma}+|t-s|^{\frac\gamma2}}.
\end{equation*}
This together with (\ref{Sept3}) leads to (\ref{Kop1a}) with
$\varepsilon=1-\gamma$.
\end{proof}

Combining Lemmas \ref{L1a} and \ref{L1b}, we arrive at
\begin{kor}

Let $u$ satisfy the assumptions of {\rm Lemma \ref{L1b}}. If
$\alpha_1\ge 2$ then for arbitrary small $\varepsilon>0$
\begin{equation}\label{Oct2}
x_n^{\alpha_1-2+\varepsilon}|D^\alpha u|\le
\frac{C}{R^{|\alpha|-\alpha_1+2-\varepsilon}}\sup_{Q_{R}^+(x^0,t^0)}
u\qquad \mbox{in}\quad Q_{R/8^{|\alpha|}}^+(x^0,t^0),
\end{equation}
where $C$ depends on $\nu$, $\alpha$ and $\varepsilon$.
\end{kor}

Now let us turn to estimating of derivatives of the Green function.

\begin{lem}\label{LTT2}
The following estimate for the Green function is valid for $ s<t$:
\begin{equation}\label{May1}
|D_x^\alpha D_y^\beta\Gamma^{\cal D}(x,y;t;s)|\le C\,(t-s)^{-\frac
{n+|\alpha|+|\beta|}2}
\cdot\exp\left(-\frac{\sigma|x-y|^2}{t-s}\right),
\end{equation}
where the positive constant $\sigma$ depends only on the
ellipticity constant $\nu$ and $C$ may depend on $\nu$, $\alpha$
and $\beta$, provided one of the following four conditions is
fulfilled:

{\rm (i)}   $\alpha$ and $\beta$ are arbitrary,  and
$x_n\ge\sqrt{(t-s)/8}$, $y_n\ge\sqrt{(t-s)/8}$;

{\rm (ii)}  $\alpha$ and $\beta$  satisfy
 $\alpha_1\le1$ and $\beta_1\le1$ respectively and $x,y\in \mathbb R^n_+$;

{\rm (iii)} $\beta$ is arbitrary, $\alpha$ satisfies
 $\alpha_1\le 1$ and  $y_n\ge\sqrt{(t-s)/8}$;

{\rm (iv)}  $\alpha$ is arbitrary,  $\beta$ satisfies
 $\beta_1\le 1$ and  $x_n\ge\sqrt{(t-s)/8}$.

\end{lem}

\begin{proof}
It is sufficient to prove the estimate for $s=0$.

 Let $|\alpha|=|\beta|=0$. Then estimate (\ref{May1}) is a
consequence of estimate (\ref{Okt1}).

(i) Let $\beta=0$. First, we suppose that $x_n\ge 1/2$. Using
Lemma \ref{L1a}(i) and estimate (\ref{May1}) for
$|\alpha|=|\beta|=0$, we obtain
\begin{equation}\label{Aug1}
|D_x^\alpha\Gamma^{\cal D}(x,y,1,0)|\le
C\sup_{Q_{1/2}(x,1)}|\Gamma^{\cal D}(\cdot,y,\cdot,0)|\le C\exp
(-\sigma|x-y|^2).
\end{equation}
Now, estimate (\ref{May1}) for $x_n\ge\sqrt{t/8}$ follows by
homogeneity.

Since the Green function is symmetric, we obtain also estimate
(\ref{May1}) in the case $\alpha=0$ and $\beta$ is arbitrary.

To prove (\ref{May1})\index{} in general case, we consider the
function $G_\beta(x,y,t)=D_y^\beta \Gamma^{\cal D}(x,y,t,0)$.
Reasoning as above we arrive at estimate (\ref{Aug1}) with
$\Gamma^{\cal D}$ replaced by $G_\beta$. Certainly at the last
step we must use (\ref{May1}) with $\alpha=0$ which is already
proved. So, the case (i) is completed.

(ii) Let  first $\beta=0$. By homogeneity it suffices to prove
(\ref{May1}) for $t=1$ . Using Lemma \ref{L1a}(ii) and estimate
(\ref{May1}) for $|\alpha|=|\beta|=0$, we obtain  estimate
(\ref{Aug1}), which implies (\ref{May1}) for $\beta=0$. Since the
Green function is symmetric with respect to $x$ and $y$, we obtain
also estimate (\ref{May1}) for $\alpha=0$. In order to handle the
general case we apply Lemma \ref{L1a}(ii) to the function
$D_y^\beta\Gamma^{\cal D}(x,y,t,0)$ and using estimate
(\ref{May1}) for $\alpha=0$, we obtain (\ref{Aug1}) with
$\Gamma^{\cal D}$ replaced by $D_y^\beta\Gamma^{\cal D}$. By
homogeneity of the Green function we arrive at (\ref{May1}).

The cases (iii) and (iv) are considered similarly.
\end{proof}

\medskip

Below we use the notations
$$
{\cal R}_x=\frac{x_n}{x_n+\sqrt{t-s}},\qquad {\cal
R}_y=\frac{y_n}{y_n+\sqrt{t-s}}.
$$

\begin{sats}\label{T2}
For $x,y\in\mathbb R_+^n$ and $ s<t$ the following estimate is
valid:
\begin{equation}\label{May3}
|D_{x}^{\alpha}D_{y}^{\beta} \Gamma^{\cal D}(x,y;t;s)|\le
C\,\frac{{\cal R}_x^{2-\alpha_1-\varepsilon} {\cal
R}_y^{2-\beta_1-\varepsilon}}{(t-s)^{\frac {n+|\alpha|+|\beta|}2}}
\,\exp \left(-\frac{\sigma|x-y|^2}{t-s}\right),
\end{equation}
where $\sigma$ is the same as in {\rm Lemma \ref{LTT2}},
$\varepsilon$ is an arbitrary small positive number  and $C$ may
depend on $\nu$, $\alpha$, $\beta$ and $\varepsilon$. If
$\alpha_1\le 1$ {\rm(}or $\beta_1\le 1${\rm)} then
$2-\alpha_1-\varepsilon$ {\rm(}$2-\beta_1-\varepsilon${\rm)} must
be replaced by $1-\alpha_1$ {\rm(}$1-\beta_1${\rm)} respectively
in the corresponding exponents.
\end{sats}

\begin{proof}
It is sufficient to prove the estimate for $s=0$.

First let us prove the estimate
\begin{equation}\label{May3p}
|D_{x}^{\alpha}\Gamma^{\cal D}(x,y;t;0)|\le C\,\frac{{\cal
R}_x^{2-\alpha_1-\varepsilon}}{t^{\frac{n+|\alpha|}2}}
\,\exp\left(-\frac{\sigma|x-y|^2}{t}\right)
\end{equation}
for $\alpha_1\ge 2$ and
\begin{equation}\label{May3pp}
|D_{x}^{\alpha}\Gamma^{\cal D}(x,y;t;0)|\le C\,\frac{{\cal
R}_x^{1-\alpha_1}}{t^{\frac{n+|\alpha|}2}}\,\exp
\left(-\frac{\sigma|x-y|^2}{t}\right)
\end{equation}
for $\alpha_1\le 1$.

Consider the case  $\alpha_1\ge 2$. If $x_n\le 1/2$ then using
estimate (\ref{Oct2}) we obtain
\begin{multline*}
|D_x^\alpha\Gamma^{\cal D}(x,y;1,0)|\le
Cx_n^{2-\alpha_1-\varepsilon}
\sup_{Q_{1/2}^+(x,1)}|\Gamma^{\cal D}(\cdot,y;\cdot,0)|\le\\
\le Cx_n^{2-\alpha_1-\varepsilon}\exp (-\sigma|x-y|^2),
\end{multline*}
which implies (\ref{May3p}) for $x_n\le \sqrt{t/8}$  by
homogeneity. If $x_n>\sqrt{t/8}$ estimate (\ref{May3p}) follows
from Lemma \ref{LTT2}(iv). If $\alpha_1=1$  estimate
(\ref{May3pp}) follows from Lemma \ref{LTT2}(ii). It remains to
consider  the case $\alpha_1=0$. If $x_n\ge\sqrt{t/8}$ then
estimate (\ref{May3pp}) follows from Lemma \ref{LTT2}(ii). Let
$x_n<\sqrt{1/8}$. Then
\begin{eqnarray*}
&&|\Gamma^{\cal D}(x,y;1,0)|=|\int\limits_0^{x_n}D_\tau\Gamma^{\cal
D}(x',\tau,y;1,0)d\tau
|\\
&&\le C\int\limits_0^{x_n}\exp (-\sigma|\tau-y_n|^2)d\tau\exp
(-\sigma|x'-y'|^2)\le Cx_n\exp (-\sigma|x-y|^2),
\end{eqnarray*}
where we applied estimate (\ref{May3pp}) with $\alpha_1=1$. Using
again the homogeneity argument, we arrive at (\ref{May3pp}) for
$\alpha_1=0$.

Reference to the symmetry of the Green function implies
(\ref{May3}) for $\alpha=0$ from (\ref{May3p}) and (\ref{May3pp}).
Now repeating the proof of Lemma \ref{LTT2} but using inequality
(\ref{May3}) with $\alpha=0$ instead of (\ref{May1}) with
$\alpha=0$ we arrive at the estimate
\begin{equation}\label{May1p}
|D_x^\alpha D_y^\beta\Gamma^{\cal D}(x,y;t,s)|\le C\,\frac{{\cal
R}_y^{2-\beta_1-\varepsilon}}{(t-s)^{\frac{n+|\alpha|+|\beta|}2}}\,\exp
\left(-\frac{\sigma|x-y|^2}{t-s}\right)
\end{equation}
in the cases (i) $\alpha_1\le 1$ and (ii) $\alpha$ is arbitrary
and $x_n\ge\sqrt{t/8}$. Moreover, $2-\beta_1-\varepsilon$ must be
replaced by $1-\beta_1$ when $\beta_1\le 1$.

Finally, repeating the proof of estimates (\ref{May3p}) and
(\ref{May3pp}) but using inequality (\ref{May1p}) instead of
(\ref{May1}) we arrive at (\ref{May3}).
\end{proof}

\section{The weighted estimates in a half-space}\label{Rn+2}

The main result of this section, which is equivalent to estimate
(\ref{Jan5}) and an analogous estimate for the $\widetilde
L_{p,q}$-norms, is Theorem \ref{halfspace}. We precede it by the
following three lemmas which constitute main steps in its proof.

\begin{lem}
Let $\mu\in \mathbb R$, $s<t$ and $x_n>\sqrt{t-s}$,
$y_n>\sqrt{t-s}$. Then the following estimates are valid:
\begin{multline}\label{calG}
\left|\frac {x_n^{\mu}}{y_n^{\mu}}\, D^2_x\Gamma^{\cal D}(x,y;t,s)-
D^2_x\Gamma(x,y;t,s)\right|\\
\le C\frac { y_n^{-1}} {(t-s)^{\frac {n+1}2}} \exp
\left(-\frac{\sigma_0|x-y|^2}{t-s} \right),
\end{multline}
\begin{multline}\label{D2y_calG}
\left|\frac {x_n^{\mu}}{y_n^{\mu}}\, D^2_xD^2_y\Gamma^{\cal
D}(x,y;t,s)-
D^2_xD^2_y\Gamma(x,y;t,s)\right|\\
\le C\frac {y_n^{-1}}{(t- s)^{\frac {n+3}2}} \exp
\left(-\frac{\sigma_0|x-y|^2}{t- s} \right)
\end{multline}
and
\begin{multline}\label{partials_calG}
\left|\frac {x_n^{\mu}}{y_n^{\mu}}\, D^2_x\partial_s\Gamma^{\cal D}(x,y;t,s)-
D^2_x\partial_s\Gamma(x,y;t,s)\right| \\
\le C\frac { y_n^{-1}}{(t- s)^{\frac {n+3}2}} \exp
\left(-\frac{\sigma_0|x-y|^2}{t- s} \right),
\end{multline}
where the positive constant $\sigma_0$ depends only on $\nu$ and
$C$ may depend on $\nu$ and $\mu$.
\end{lem}

\begin{proof} It is sufficient to prove Lemma for $s=0$. We put
$$
{\mathbb G}_{\alpha,\beta}(x,y;t)=\frac {x_n^{\mu}}{y_n^{\mu}} \,
D_{x}^\alpha D_{y}^\beta\Gamma^{\cal D}(x,y;t,0)- D_{x}^\alpha
D_{y}^\beta\Gamma(x,y;t,0)
$$
Since the functions ${\mathbb G}_{\alpha,\beta}$ are positively
homogeneous with respect to variables $x$, $y$ and $\sqrt{t}$, it
is sufficient to prove Lemma for $t=1$ and correspondingly for
$x_n>1$ and $y_n>1$. First, let us prove the estimate
\begin{equation}\label{Nov30}|{\mathbb G}_{0,0}(x,y;1)|
\le Cy_n^{-1} \exp\left(-\widetilde\sigma|x-y|^2 \right)
\end{equation}
for $x_n>1/2$ and $y_n>1/2$. Here $\widetilde\sigma$ is a positive
constant depending on $\sigma$. Let $x^0$ be arbitrary point with
$x^0_n>1/2$. By $\zeta=\zeta(\rho)$ denote a smooth function such
that $\zeta(\rho)=1$ for $\rho\leq 1/4$ and $\zeta(\rho)=0$ for
$\rho\geq 1/2$.
 Applying the operator ${\cal L}_0$ to the function
 $$\phi(x,y,t)= \zeta(|x-x^0|/x^0_n){\mathbb G}_{0,0}(x,y;t),$$
 we obtain
 \begin{equation}\label{Nov1}
 {\cal L}_0\phi(x,y,t)=F_1(x,y,t)+F_2(x,y,t),
 \end{equation}
 where
$$
F_1=-2a^{kj}D_{x_j}\Big (\zeta\Big(\frac{|x-x^0|}{x^0_n}\Big)
 \frac{x_n^{\mu}}{y_n^{\mu}}\Big)D_{x_k}\Gamma^{\cal D}
 -a^{kj}D_{x_k}D_{x_j}
 \Big (\zeta\Big(\frac {|x-x^0|}{x^0_n}\Big)\frac {x_n^{\mu}}{y_n^{\mu}}\Big)\Gamma^{\cal D}
$$
 and
 $$
 F_2=2a^{kj}D_{x_j}\zeta\Big(\frac {|x-x^0|}{x^0_n}\Big)D_{x_k}
 \Gamma+a^{kj}D_{x_k}D_{x_j}\zeta\Big(\frac {|x-x^0|}{x^0_n}\Big)\Gamma.
 $$
 Solving (\ref{Nov1}), we arrive at
 \begin{equation}\label{Nov1g}
 {\mathbb G}_{0,0}(x^0,y;1)=\int\limits_0^1\int\limits_{\mathbb R^n_+}
 \Gamma^{\cal D}(x^0,z;1,s) \Big (F_1(z,y,s)+F_2(z,y,s) \Big )dzds.
 \end{equation}
 In what follows we'll write $x$ instead of $x^0$. Taking into account (\ref{May1})
 and $x_n>1/2>\sqrt{s}/2$, we estimate the first term of the integrand in (\ref{Nov1g}) by
 \begin{equation}\label{Nov5}
 C \frac{z_n^{\mu}y_n^{-\mu}x_n^{-1}}{(1-s)^{\frac n2}s^{\frac {n+1}2}}
 \exp\Big(-\frac{\sigma |x-z|^2}{1-s}-\frac{\sigma |z-y|^2}{s}\Big ).
 \end{equation}
 Similarly,  using (\ref{May0}) and (\ref{May1}) the second term can be estimated by
 \begin{equation}\label{Nov5a}
 C \frac{x_n^{-1}}{(1-s)^{\frac n2}s^{\frac {n+1}2}}\exp\Big(-\frac{\sigma |x-z|^2}{1-s}-\frac{\sigma |z-y|^2}{s}\Big ).
 \end{equation}

 Since the integration in (\ref{Nov1g}) is taken over ${|z-x|<x_n/2}$, we
 obtain $|{\mathbb G}_{0,0}(x,y;1)|$ is majorized by
 \begin{equation}\label{Nov5b}
  C\Big(\frac {x_n^{\mu-1}}{y_n^\mu}+x_n^{-1}\Big)
 \int\limits_{0}^1
 \int\limits_{\mathbb R^n_+}\exp\Big(-\frac{\sigma |x-z|^2}{1-s}-
 \frac{\sigma |z-y|^2}{s}\Big)\frac {dzds}{(1-s)^{\frac n2}s^{\frac {n+1}2}}.
 \end{equation}
 We observe that the exponent in the right-hand side does not exceed
 $$
 \exp\Big(-\frac{\sigma |x-y|^2}{2}-\frac{\sigma |x-z|^2}{2(1-s)}-\frac{\sigma |z-y|^2}{2s}\Big)
 $$
 and split the integral with respect to $s$ into two integrals, one from $0$ to $1/2$ and another from $1/2$ to $1$.
 Using the change of variables $u=(z-y)s^{-1/2}$ in the first integral and $v=(x-z)(1-s)^{-1/2}$ in the second one,
 we estimate the integral in (\ref{Nov5b}) by $C\exp\big (-\sigma|x-y|^2/2 \big)$, that
 gives the estimate
$$
|{\mathbb G}_{0,0}(x,y;1)|\le C\Big(\frac
{x_n^{\mu-1}}{y_n^\mu}+x_n^{-1}\Big)\exp\Big(-\frac{\sigma
|x-y|^2}{2}\Big ).
$$
Using that for any $\lambda\in\mathbb R$, $a>0$ and $x_n> 1/2$,
$y_n> 1/2$
\begin{equation}\label{Dec1m}
x_n^\lambda y_n^{-\lambda}\leq
C_{\lambda,a}\exp\Big(a|x_n-y_n|^2\Big ),
\end{equation}
 we arrive at (\ref{Nov30}).\medskip

Next step includes the following local estimate for solutions  to
the equation ${\cal L}_0 u=h$ in $Q_R(x_0,t_0)$:
\begin{equation}\label{Dec1}
\sup_{Q_{R/2}(x_0,t_0)}|D_x u(x,t)|\leq
C\big(R^{-1}\sup_{Q_R(x_0,t_0)}|u(x,t)|+R\sup_{Q_R(x_0,t_0)}|h(x,t)|\big).
\end{equation}
For $R=1$ it follows from the integral representation (\ref{TTN1})
and estimate (\ref{May0}) for the Green function $\Gamma$ after
rewriting equation for $u$ as equation in the whole space by
introducing an appropriate cut-off function. For arbitrary $R$ it
is proved by homogeneity arguments. Differentiating the equation
with respect $x$ and iteratively using (\ref{Dec1}), we arrive at
\begin{multline}\label{Dec1a}
\sup_{Q_{R/2^{|\alpha|}}(x_0,t_0)}|D_x^\alpha u(x,t)|\leq
C\Big(R^{-|\alpha|}\sup_{Q_R(x_0,t_0)}|u(x,t)\\
+\sum_{\beta<\alpha}R^{2-|\alpha|+|\beta|}\sup_{Q_R(x_0,t_0)}|D_x^\beta
h(x,t)|\Big).
\end{multline}
Applying (\ref{Dec1a}) with $|\alpha|\leq 2$, $t_0=1$ and $R=1/4$,
to equation ${\cal L}_0{\mathbb G}_{0,0}=h$, where
$$
h(x,y,t)=-2\mu\frac{x_n^{\mu-1}}{y_n^\mu}a^{jn}D_{x_j}
\Gamma^{\cal D}(x,y;t,0)-a^{nn}\mu(\mu-1)
\frac{x_n^{\mu-2}}{y_n^\mu}\Gamma^{\cal D}(x,y;t,0)
$$
(cf. (\ref{Nov1})). This gives
$$
|D_x^\alpha{\mathbb G}_{0,0}(x,y;1)|\leq C\sup_{ Q_{1/4}(x,1)}
\big(|{\mathbb G}_{0,0}(\cdot,y;\cdot)|+|h(\cdot,y,\cdot)|
+|D_xh(\cdot,y,\cdot)|\big)
$$
(the last term must be omitted if $|\alpha|=1$). Using
(\ref{Nov30}) for estimating the first term in the right-hand side
and (\ref{May1}) for estimating the other terms, together with
homogeneity arguments, we arrive at the estimate
\begin{equation}\label{Nov30a}
|D_x^\alpha{\mathbb G}_{0,0}(x,y;t)|  \le C\frac { 1} {t^{\frac
{n+|\alpha|-1}2}}\frac {x_n^{\mu-1}}{y_n^\mu} \exp
\left(-\widetilde\sigma_1\frac{|x-y|^2}{t} \right)
\end{equation}
for $x_n>3/4$ and $y_n>3/4$ with a certain positive
$\widetilde\sigma_1$ depending on $\sigma$. Expressing ${\mathbb
G}_{\alpha,0}$ in terms of derivatives of ${\mathbb G}_{0,0}$ and
using (\ref{Dec1m}), we arrive at (\ref{calG}).

Let us prove (\ref{D2y_calG}). Since the Green function is
symmetric with respect to  $x$ and $y$, the estimate
\begin{equation}\label{Nov30bc}
|{\mathbb G}_{0,\beta}(x,y;t)|\leq C\frac { x_n^{-1}} {t^{\frac
{n+1}2}} \exp \left(-\widetilde\sigma_1\frac{|x-y|^2}{t} \right)
\end{equation}
holds for $|\beta |=2$, $x_n>3\sqrt{t}/4$ and $y_n>3\sqrt{t}/4$.
Applying the local estimate (\ref{Dec1a}) with $|\alpha|\leq 2$ to
the equation ${\cal L}_0{\mathbb G}_{0,\beta}=h_\beta$, where
\begin{multline*}
h_\beta(x,y,t)\\
=-2\mu\frac{x_n^{\mu-1}}{y_n^\mu}a^{kn}D_{x_k}
D_y^\beta\Gamma^{\cal
D}(x,y;t,0)-a^{nn}\mu(\mu-1)\frac{x_n^{\mu-2}}{y_n^\mu}
D_y^\beta\Gamma^{\cal D}(x,y;t,0),
\end{multline*}
we obtain
$$
|D_{x}^\alpha{\mathbb G}_{0,\beta}(x,y;1)|\leq C
\sup_{Q_{1/4}(x,1)}\big( |{\mathbb
G}_{0,\beta}(\cdot,y;\cdot)|+|h_\beta(\cdot,y,\cdot)|+
|D_xh_\beta(\cdot,y,\cdot)|\big)
$$
(the last term must be omitted if $|\alpha|=1$). Using here
estimates (\ref{Nov30bc}) and (\ref{May1}) together with
homogeneity arguments and (\ref{Dec1m}), we arrive at
$$
|D_{x}^\alpha{\mathbb G}_{0,\beta}(x,y;1)|\leq C\frac { y_n^{-1}}
{t^{\frac {n+|\alpha|+1}2}} \exp
\left(-\widetilde\sigma_2\frac{|x-y|^2}{t} \right)
$$
for $x_n>1$ and $y_n>1$, which implies (\ref{D2y_calG}). Finally,
inequality (\ref{partials_calG}) follows from (\ref{D2y_calG}),
since the derivative with respect to $s$ can be expressed through
the second derivatives with respect to $y$. The proof is complete.
\end{proof}

For $\mu\in \mathbb R$ we define the weighted kernels
$${\cal G}_{ij}(x,y;t, s)=\frac {x_n^{\mu}}{y_n^{\mu}}\,
D_{x_i}D_{x_j}\Gamma^{\cal D}(x,y;t, s)-\chi_{\{x_n>\sqrt{t-
s}\}}\, D_{x_i}D_{x_j}\Gamma(x,y;t, s).
$$

\begin{lem}\label{LTT3}
The following estimates are valid:
\begin{equation}\label{calGij}
\left|{\cal G}_{ij}(x,y;t, s)\right| \le C\,\frac {{\cal
R}^{1-\varepsilon}_{x}{\cal R}_{y}} {(t- s)^{\frac {n+1}2}}\,\Big
(\frac {x_n^{\mu-1}}{y_n^{\mu}}+y_n^{-1}\Big )\, \exp
\left(-\frac{\sigma_1|x-y|^2}{t- s} \right)
\end{equation}
and
\begin{equation}\label{partials_calGij}
\left|\partial_s {\cal G}_{ij}(x,y;t, s)\right| \le C\,\frac {
{\cal R}^{1-\varepsilon}_{x}{\cal
R}^{-\varepsilon}_{y}}{(t-s)^{\frac {n+3}2}} \,\Big (\frac
{x_n^{\mu-1}}{y_n^{\mu}}+y_n^{-1}\Big)\, \exp
\left(-\frac{\sigma_1|x-y|^2}{t- s} \right)
\end{equation}
for $\mu\in \mathbb R$ and $s<t$.  Here $\varepsilon$ is an
arbitrary small positive number, the positive constant $\sigma_1$
depends only on $\nu$ while $C$ may depend on $\nu$, $\mu$ and
$\varepsilon$.
\end{lem}

\begin{proof} Let $x_n>\sqrt{t-s}$ and $y_n>\sqrt{t-s}$. Then
${\cal R}_{x}\asymp1$ and ${\cal R}_{y}\asymp1$, where ${\cal
R}_{x}\asymp1$ means that ${\cal R}_{x}$ is estimated from below
and from above by positive constants independent of $x$, $y$, $t$
and $s$. Therefore, (\ref{calGij}) and (\ref{partials_calGij})
follow from (\ref{calG}) and (\ref{partials_calG}) respectively.

Now let $x_n<\sqrt{t-s}$ and $y_n>0$. Then
$$
{\cal G}_{ij}(x,y;t, s)=\frac {x_n^{\mu}}{y_n^{\mu}}
D_{x_i}D_{x_j}\Gamma^{\cal D}(x,y;t, s)
$$
and (\ref{May3}) implies
$$
|{\cal G}_{ij}(x,y;t, s)|\leq C\frac {{\cal
R}^{-\varepsilon}_{x}{\cal R}_{y}} {(t- s)^{\frac {n+2}2}}\frac
{x_n^{\mu}}{y_n^{\mu}} \exp \left(-\frac{\sigma|x-y|^2}{t- s}
\right).
$$
Since $x_n/\sqrt{t-s}\leq C{\cal R}_x$ in this case, the last
inequality implies (\ref{calGij}). Using the same arguments, we
estimate
$$
\frac {x_n^{\mu}}{y_n^{\mu}} D_{x_i}D_{x_j}D^2_{y}\Gamma^{\cal
D}(x,y;t, s)
$$
by the right-hand side in (\ref{partials_calGij}). Since the
derivative with respect to $s$ can be expressed through the second
derivatives with respect to $y$, we obtain
(\ref{partials_calGij}).

Finally consider the case $x_n>\sqrt{t-s}$ and $y_n<\sqrt{t-s}$.
Using estimates (\ref{May0}) and (\ref{May3}), we have
$$
|{\cal G}_{ij}(x,y;t, s)|\leq C\frac {1} {(t- s)^{\frac
{n+2}2}}\Big (\frac {x_n^{\mu}{\cal R}_{y}}{y_n^{\mu}}+1\Big )
\exp \left(-\frac{\sigma|x-y|^2}{t- s} \right).
$$
We observe that ${\cal R}_{x}\asymp1$ and
$$
\frac {x_n^{\mu}{\cal R}_{y}}{y_n^{\mu}}+1\leq x_n{\cal
R}_{y}\,\Big(\frac {x_n^{\mu-1}}{y_n^{\mu}}+2y_n^{-1}\Big ).
$$
Since (\ref{Dec1m}) implies
$$
\frac{x_n}{\sqrt{t-s}}\leq C_a\exp\Big(a
\frac{|x_n-\sqrt{t-s}|^2}{t-s}\Big )\leq C_a\exp\Big(a
\frac{|x_n-y_n|^2}{t-s}\Big )
$$
for every positive $a$, we arrive at (\ref{calGij}) with a
$\sigma_1$ less that $\sigma$.

Similar arguments estimate the function
$$
\frac {x_n^{\mu}}{y_n^{\mu}}\, D_{x_i}D_{x_j}D^2_{y}\Gamma^{\cal
D}(x,y;t,s)- D_{x_i}D_{x_j}D^2_{y}\Gamma(x,y;t, s).
$$
by the right-hand side in (\ref{partials_calGij}). Since the
derivative with respect to $s$ can be expressed through the second
derivatives with respect to $y$, we obtain
(\ref{partials_calGij}). The proof is completed.
\end{proof}

\begin{lem}\label{weak_2}
Let a function $h$ be supported in the layer $|s-s^0|\le\delta$
and satisfy $\int h(y,s)\ ds\equiv 0$. Also let $p\in (1,\infty)$
and $\mu$ be subject to {\rm (\ref{Mart2a})}. Then the integral
operator ${\cal G}_{ij}$ satisfies
$$\int\limits_{|t- s^0|>2\delta}\Vert ({\cal G}_{ij}h)(\cdot, t)\Vert_p\ dt\le
C\,\Vert h\Vert_{p,1},$$

\noindent where $C$ does not depend on $\delta$ and $ s^0$.
\end{lem}

\begin{proof}
By $\int h(y, s)\ ds\equiv 0$, we have
\begin{equation}\label{difference}
({\cal G}_{ij}h)(x,t)=\int\limits_{0}^{t}\int\limits_{\mathbb
R^n_+} \Bigl({\cal G}_{ij}(x,y;t, s)-{\cal G}_{ij}(x,y;t,
s^0)\Bigr)\, h(y, s)\ dy\ ds.
\end{equation}
We choose $\varepsilon>0$ such that
\begin{equation}\label{mueps}
-\frac 1p+\varepsilon<\mu<2-\frac 1p-\varepsilon.
\end{equation}
For $| s- s^0|<\delta$ and $t- s^0>2\delta$, estimates
(\ref{partials_calGij}) and (\ref{May0}) with $|\alpha|=2$,
$|\beta|=0$ imply
\begin{eqnarray*}
&&\left|{\cal G}_{ij}(x,y;t, s)-{\cal G}_{ij}(x,y;t, s^0)\right|
\le\int\limits_{s^0}^s|\partial_\tau {\cal G}_{ij}(x,y;t,
\tau)|\,d\tau\\
&&+
|D_{x_i}D_{x_j}\Gamma(x,y;t,t-x_n^2)|\,\chi_{\{x_n^2\in(t-s,t-s^0)\}}
\le C\,\frac {{\cal R}^{1-\varepsilon}_{x} {\cal
R}^{-\varepsilon}_{y}} {(t- s)^{\frac {n+1}2}} \,\Big (\frac
{x_n^{\mu-1}}{y_n^{\mu}}+y_n^{-1}\Big)\\
&&\times \frac {\delta} {t- s} \exp \left(-\frac
{\sigma|x-y|^2}{t- s}\right)+C\,\frac
{\chi_{\{x_n^2\in(t-s,t-s^0)\}}}{(t-s)^{\frac {n+2}2}}
\,\exp\left(-\frac {\sigma|x-y|^2}{t-s}\right).
\end{eqnarray*}
On the other hand, estimate (\ref{calGij}) gives
\begin{eqnarray*}
&&\left|{\cal G}_{ij}(x,y;t, s)-{\cal G}_{ij}(x,y;t, s^0)\right|\\
&&\le C\,\frac {{\cal R}^{1-\varepsilon}_{x}{\cal R}_{y}}{(t-
s)^{\frac {n+1}2}}
 \,\Big (\frac {x_n^{\mu-1}}{y_n^{\mu}}+y_n^{-1}\Big) \, \exp
\left(-\frac{\sigma|x-y|^2}{t- s} \right).
\end{eqnarray*}
Combination of these estimates gives
\begin{eqnarray*}
&&\left|{\cal G}_{ij}(x,y;t, s)-{\cal G}_{ij}(x,y;t, s^0)\right| \\
&&\le C\,\frac {{\cal R}^{1-\varepsilon}_{x} {\cal
R}^{1-\varepsilon}_{y}} {(t- s)^{\frac {n+1}2}}\,
 \,\Big (\frac {x_n^{\mu-1}}{y_n^{\mu}}+y_n^{-1}\Big)
\,\left(\frac {\delta}{t- s}\right)^{\frac
{\varepsilon}{1+\varepsilon}}
\!\!\!\! \, \exp \left(-\frac {\sigma|x-y|^2}{t- s} \right)\\
&&+C\,\frac {\chi_{\{x_n^2\in(t-s,t-s^0)\}}}{(t-s)^{\frac {n+2}2}}
\,\exp\left(-\frac {\sigma|x-y|^2}{t-s}\right) =: {\cal
J}_1(x,y,t, s)+{\cal J}_2(x,y,t, s).
\end{eqnarray*}
Applying this inequality for estimating the right-hand side in
(\ref{difference}), we obtain
$$\int\limits_{|t- s^0|>2\delta}\Vert ({\cal G}_{ij}h)(\cdot, t)\Vert_p\ dt
\le \int\limits_{|t- s^0|>2\delta}\Vert ({\cal J}_1h)(\cdot,
t)\Vert_p\ dt +\int\limits_{|t- s^0|>2\delta}\Vert ({\cal
J}_2h)(\cdot, t)\Vert_p\ dt.$$ The second term is estimated by
$C\Vert h\Vert_{p,1}$ in the proof of Theorem \ref{trunc}, Step~3.
Further, the first term can be treated by Lemma \ref{L_p_1} with
$m=1$, $r=1$, $\lambda_1=-\varepsilon$, $\lambda_2=1-\varepsilon$,
$\varkappa=\frac {\varepsilon}{1+\varepsilon}$. The inequality
(\ref{mueps}) becomes (\ref{mu_m}), and $y_n^{-1}$ corresponds to
a particular case $\mu=1$. Thus, this term is also estimated by
$C\Vert h\Vert_{p,1}$.
\end{proof}

Now we are in position to prove one of the main results of this paper.
\begin{sats}\label{halfspace}
Let $p,\ q\in\ (1,\infty)$ and $\mu$ be subject to {\rm
(\ref{Mart2a})}. Then a solution of {\rm (\ref{Jan1})} in $\mathbb
R^n_+\times \mathbb R$ with  zero Dirichlet condition satisfies
\begin{equation}\label{est_halfspace}
\gathered |\!|\!|x_n^\mu\partial_t u|\!|\!|_{p,q}+|\!|\!|x_n^\mu
D^2u|\!|\!|_{p,q}\le
C\ |\!|\!|x_n^\mu f|\!|\!|_{p,q},\\
\|x_n^\mu\partial_t u\|_{p,q}+\|x_n^\mu D^2u\|_{p,q}\le C\
\|x_n^\mu f\|_{p,q},
\endgathered
\end{equation}
where $C$ depends only on $\nu$, $\mu$, $p$ and $q$.
\end{sats}

\begin{proof}
The estimate of the last terms in the left-hand side of
(\ref{est_halfspace}) is equivalent to the boundedness of integral
operators with kernels
$${\mathfrak G}_{ij}^{\cal D}(x,y;t,s)=\frac
{x_n^\mu} {y_n^\mu}D_{x_i}D_{x_j}\Gamma^{\cal D}(x,y;t,s)$$ in
$\widetilde L_{p,q}(\mathbb R^n_+\times \mathbb R_+)$ and
$L_{p,q}(\mathbb R^n_+\times \mathbb R)$, respectively.

First, we consider the case $j\ne n$. The kernel ${\mathfrak
G}_{ij}^{\cal D}(x,y;t,s)$ can be written as
$${\mathfrak G}_{ij}^{\cal D}(x,y;t,s)={\cal G}_{ij}(x,y;t,s)+
\chi_{\{x_n>\sqrt{t- s}\}}\, D_{x_i}D_{x_j}\Gamma(x,y;t, s).$$ By
Theorem \ref{trunc} the operator corresponding to the second term
is bounded both in $\widetilde L_{p,q}(\mathbb R^n\times \mathbb
R)$ and in $L_{p,q}(\mathbb R^n\times \mathbb R)$ spaces.

Estimate (\ref{calGij}) shows that the operator ${\cal G}_{ij}$
satisfies the assumptions of Lemmas \ref{L_p} and \ref{L_p_infty}
with $m=1$, $r=1$, $\lambda_1=-\varepsilon$ and $\lambda_2=1$.
(we recall that the term $y_n^{-1}$ corresponds a particular case
$\mu=1$)
Therefore, under condition (\ref{mueps}) this operator is bounded
in $L_p(\mathbb R^n_+\times\mathbb R)$ and in $\widetilde
L_{p,\infty}(\mathbb R^n_+\times \mathbb R)$.
Since $\varepsilon$ is
arbitrarily small, this is true under condition (\ref{Mart2a}).
Generalized Riesz--Thorin theorem, see, e.g., \cite[1.18.7]{Tr},
shows that the operator ${\cal G}_{ij}$ is bounded in $\widetilde
L_{p,q}(\mathbb R^n_+\times \mathbb R_+)$ for any $q\ge p$. For
$q<p$ the statement follows by duality arguments.

Further, by Lemma \ref{weak_2}, the operator ${\cal G}_{ij}$
satisfies the assumptions of Theorem 3.8 in \cite{BIN}. Therefore,
this operator is bounded in $L_{p,q}(\mathbb R^n_+\times \mathbb
R)$ for any $q\in\,(1,p]$. For $q>p$ the statement follows by
duality arguments.

Finally, to estimate $\partial_t u$ and $D_nD_nu$, we rewrite the
equation (\ref{Jan1}) as
\begin{equation}\label{heat}
\partial_tu-a^{nn}\Delta u=\widetilde a^{ij}D_iD_ju+f,
\end{equation}
where $\widetilde a^{ij}(t)=a^{ij}(t)-a^{nn}(t)\delta^{ij}$. After
the change of variable $\tau=\int_0^ta^{nn}(s)\,ds$, equation
(\ref{heat}) becomes
$$
\partial_\tau u-\Delta u=\widetilde f,
$$
where
$$
|\!|\!|x_n^\mu\widetilde f|\!|\!|_{p,q}\le C\ |\!|\!|x_n^\mu
f|\!|\!|_{p,q}, \qquad \|x_n^\mu\widetilde f\|_{p,q}\le C\
\|x_n^\mu f\|_{p,q}.
$$
Now  estimate (\ref{est_halfspace}) follows from \cite[Theorem
7.6]{Na}.
\end{proof}

\section{Solvability of linear and quasilinear Dirichlet problems}\label{solv}

Let $\Omega$ be a bounded domain in $\mathbb R^n$ with boundary
$\partial\Omega$. For a cylinder $Q=\Omega\times(0,T)$, we denote by
$\partial'Q=\{\partial\varOmega\times(0,T)\}\cup
\{\overline{\varOmega} \times \{0\}\}$ its parabolic boundary.

We introduce two scales of functional spaces: ${\mathbb L}_{p,q,(\mu)}(Q)$ and
$\widetilde{\mathbb L}_{p,q,(\mu)}(Q)$,
with norms
$$
{\pmb \|}f{\pmb \|}_{p,q,(\mu),Q}=
\Vert (\widehat d(x))^{\mu}f\Vert_{p,q,Q}=
\Big (\int\limits_0^T\Big (\int\limits_\Omega
(\widehat d(x))^{\mu p}|f(x,t)|^pdx\Big )^{q/p}dt\Big )^{1/q}
$$
and
$$
\pmb {|\!|\!|}f\pmb {|\!|\!|}_{p,q,(\mu),Q}=
|\!|\!| (\widehat d(x))^{\mu}f|\!|\!|_{p,q,Q}=
\Big (\int\limits_\Omega\Big
(\int\limits_0^T (\widehat d(x))^{\mu q}|f(x,t)|^qdt\Big )^{p/q}dx\Big)^{1/p}
$$
respectively, where $\widehat d(x)$ stands for the distance from
$x \in \Omega$ to $\partial\Omega$. For $p=q$ these spaces coincide, and we write
${\mathbb L}_{p,(\mu)}(Q)$.

We denote by ${\mathbb W}^{2,1}_{p,q,(\mu)} (Q)$ and $\widetilde
{\mathbb W}^{2,1}_{p,q,(\mu)} (Q)$ the set of functions with the
finite seminorms
$$
{\pmb \|}\partial_tu {\pmb \|}_{p,q,(\mu),Q}+\sum_{ij}{\pmb \|}
D_iD_ju {\pmb \|} _{p,q, (\mu),Q}
$$
and
$$
\pmb{|\!|\!|} \partial_tu \pmb{|\!|\!|}_{p,q, (\mu),Q}+\sum_{ij}
\pmb{|\!|\!|} D_iD_ju \pmb{|\!|\!|}_{p,q, (\mu),Q}
$$
respectively. These seminorms become norms on the subspaces
defined by $u|_{\partial'Q}=0$.

We say $\partial\Omega \in {\cal W}^2_{p,(\mu)}$ if for any point
$x^0\in\partial\Omega$ there exists a neighborhood $\cal U$  and a
diffeomorphism $\Psi$ mapping ${\cal U}\cap\Omega$ onto the
half-ball $B_1^+$ and satisfying
$$(\widehat d(x))^{\mu}D^2\Psi\in L_p({\cal U}\cap\Omega);\qquad
x_n^\mu D^2\Psi^{-1}\in L_p(B_1^+),$$ where corresponding norms are
uniformly bounded with respect to $x^0$.

It is well known (see, e.g., \cite{Li} and \cite[Lemma 2.6]{KyKr})
that if $\partial\Omega \in {\cal C}^{1,\delta}$, $\delta\in[0,1]$,
then $\partial\Omega \in {\cal W}^2_{\infty,(1-\delta)}$. Moreover,
in this case corresponding diffeomorphisms $\Psi,\Psi^{-1}\in{\cal
C}^{1,\delta}$. Here ${\cal C}^{1,0}$ stads for $C^1$.\smallskip

We set $\widehat\mu(p,q)=1-\frac{n}{p}-\frac{2}{q}$.

\subsection{Linear Dirichlet problem in bounded domains}

We consider the initial-boundary value problem
\begin{equation}\label{DesS4}
{\cal L}u\equiv \partial_tu-a^{ij}(x,t)D_iD_ju+b^i(x,t)D_iu=f(x,t)
\ \ {\textup{in}} \ \ Q,\quad\ u|_{\partial'Q}=0,
\end{equation}
where the leading coefficients $a^{ij}\in {\cal C}(\overline{\Omega}\to
L^\infty(0,T))$ satisfy assumptions $a^{ij}=a^{ji}$ and (\ref{Jan2}).

\begin{sats}\label{linear}
Let $1<p,q<\infty$ and $\mu \in
\big(-\frac{1}{p},2-\frac{1}{p}\big)$.
\medskip

{\bf 1}. Let
$ b^i \in {\mathbb L}_{{\overline{p}},{\overline{q}},(\overline{\mu})}(Q)+
{\mathbb L}_{\infty,(\overline{\overline{\mu}})}(Q)$, where
$\displaystyle {\overline{p}}$ and $\displaystyle {\overline{q}}$
are subject to
$$\overline{p}\ge p;\quad \left[
\begin{array}{ll}
\overline{q}=q;&
\widehat\mu(\overline{p},\overline{q})>0\\
q<\overline{q}<\infty;&
\widehat\mu(\overline{p},\overline{q})=0
\end{array}\right.,
$$
while $\overline {\mu}$ and $\overline{\overline{\mu}}$ satisfy
\begin{equation}\label{mumu}
{\overline{\mu}}=\min\{\mu, \max\{\widehat\mu(p,q),0\}\};\qquad
\overline{\overline{\mu}}\le 1, \quad \overline{\overline{\mu}}<\mu+\textstyle\frac 1p.
\end{equation}
Suppose also that either $\partial\Omega \in {\cal
W}^2_{\infty,(\overline{\overline{\mu}})}$ (in the case
$\overline{\overline{\mu}}=1$ this assumption must be replaced by
$\partial\Omega \in {\cal C}^1$) or $\partial\Omega \in {\cal
W}^2_{\overline{p},(\overline{\mu})}$. Then, for any $f \in {\mathbb
L}_{p,q,(\mu)}(Q)$, the initial-boundary value problem {\rm
(\ref{DesS4})} has a unique solution $u\in {\mathbb
W}^{2,1}_{p,q,(\mu)}(Q)$. Moreover, this solution satisfies
$${\pmb \|}\partial_tu{\pmb \|}_{p,q,(\mu)}+
\sum_{ij}{\pmb \|}D_iD_ju{\pmb \|}_{p,q,(\mu)} \le C {\pmb
\|}f{\pmb \|}_{p,q,(\mu)},
$$
where the positive constant $C$ does not depend on $f$.\medskip

{\bf 2}. Let
$ b^i \in \widetilde{\mathbb L}_{{\overline{p}},{\overline{q}},(\overline{\mu})}(Q)+
{\mathbb L}_{\infty,(\overline{\overline{\mu}})}(Q)$, where
$\displaystyle {\overline{p}}$ and $\displaystyle {\overline{q}}$ are subject to
$$\overline{q}\ge q;\quad \left[
\begin{array}{ll}
\overline{p}=p;&
\widehat\mu(\overline{p},\overline{q})>0\\
p<\overline{p}<\infty;&
\widehat\mu(\overline{p},\overline{q})=0
\end{array}\right.,
$$
while $\overline {\mu}$ and $\overline{\overline{\mu}}$ satisfy
(\ref{mumu}). Suppose also that $\partial\Omega$ satisfies the same
conditions as in the part {\bf 1}. Then, for any $f \in
\widetilde{\mathbb L}_{p,q,(\mu)}(Q)$, the initial-boundary value
problem {\rm (\ref{DesS4})} has a unique solution
$u\in\widetilde{\mathbb W}^{2,1}_{p,q,(\mu)}(Q)$. Moreover, this
solution satisfies
$$\pmb{|\!|\!|} \partial_tu \pmb{|\!|\!|}_{p,q,(\mu)}+
\sum_{ij}\pmb{|\!|\!|} D_iD_ju \pmb{|\!|\!|}_{p,q,(\mu)} \le C
\pmb{|\!|\!|} f \pmb{|\!|\!|}_{p,q,(\mu)},
$$
where the positive constant $C$ does not depend on $f$.

\end{sats}

\begin{Rem}
 These assertions generalize {\rm \cite[Theorem 4.2]{Na}} and {\rm \cite[Theorem 2.10]{KyKr}}.
\end{Rem}

\begin{proof}
The standard scheme, see \cite[Ch.IV, \S9]{LSU}, including partition of unity, local
rectifying of $\partial\Omega$ and coefficients freezing, reduces the proof to
the coercive estimates for the model problems to equation
(\ref{Jan1}) in the whole space and in the half-space. These
estimates are obtained in \cite[Theorem 1.1]{Kr} and our Theorems
1 and 4. By the H\"older inequality and the embedding theorems
(see, e.g., \cite[Theorems 10.1 and 10.4]{BIN}), the assumptions on $b^i$ guarantee that
the lower-order terms in (\ref{DesS4}) belong to desired weighted spaces,
${\mathbb L}_{p,q,(\mu)}(Q)$ and $\widetilde{\mathbb L}_{p,q,(\mu)}(Q)$, respectively.
By the same reasons, the requirements on $\partial\Omega$ imply $\partial\Omega\in {\cal C}^1$ and
ensure the invariance of assumptions on $b^i$ under rectifying of the boundary.
\end{proof}

\subsection{Quasilinear Dirichlet problem in bounded domains}

In this subsection, we consider the initial-boundary value problem
\begin{equation}\label{quasi}
\partial_tu-a^{ij}(x,t,u,Du)D_iD_ju + a(x,t,u,Du) = 0 \quad\mbox{in}\ \ Q,\qquad
u|_{\partial'Q}=0.
\end{equation}
We suppose that the first derivatives of the coefficients
$a^{ij}(x,t,z,\mathfrak p)$ with respect to $x$, $z$ and $\mathfrak
p$ are locally bounnded and the following inequalities hold for all
$(x;t) \in Q$, $z \in {\mathbb R}^1$ and $\mathfrak p \in{\mathbb
R}^n$ with some positive $\nu$ and $\nu_1$:
\begin{equation}\label{quasi1}
\gathered \nu |\xi|^2  \leqslant a^{ij}(x,t,z,\mathfrak p)\xi_i
\xi_j \leqslant
\nu^{-1}|\xi|^2\qquad  \forall \xi \in {\mathbb R}^n, \\
|a(x,t,z,\mathfrak p)| \leqslant \nu_1 |\mathfrak p|^2+
b(x,t)|\mathfrak p|+\Phi(x,t),\\
\left|\frac {\partial a^{ij}(x,t,z,\mathfrak p)}{\partial
\mathfrak p}
\right|\leqslant \frac {\nu_1}{1+|\mathfrak p|},\\
\left|{\mathfrak p}\cdot \frac {\partial a^{ij}(x,t,z,\mathfrak
p)}{\partial z}+ \frac {\partial a^{ij}(x,t,z,\mathfrak
p)}{\partial x}\right|\le \nu_1|\mathfrak p|+\Phi_1(x,t).
\endgathered
\end{equation}

\begin{sats}\label{quasilinear}
{\bf 1}. Let the following assumptions be satisfied:
\begin{description}
\item[{(i)}] $1<q\leqslant p<\infty$,\ \ $\widehat{\mu}(p,q)>0$,\
\ $-1/p<\mu<\widehat{\mu}(p,q)$,\ \ $\partial\Omega \in {\cal
W}^2_{p,(\mu)}$; \item[{(ii)}] functions $a^{ij}$ and $a$ satisfy
the structure conditions {\rm (\ref{quasi1})}; \item[{(iii)}]
$b,\Phi\in{\mathbb L}_{p,q,(\mu)}(Q)$; \item[{(iv)}] $\Phi_1 \in
{\mathbb L}_{p_1,q_1,(\mu_1)}(Q)$,\quad $q_1\leqslant
p_1<\infty$,\ \  $\widehat{\mu}(p_1,q_1)>\max\{\mu_1,0\}$;
\item[{(v)}] $a(\cdot,z,\mathfrak p)$ is continuous w.r.t.
$(z,\mathfrak p)$ in the norm ${\pmb \|}\cdot{\pmb
\|}_{p,q,(\mu),Q}$.
\end{description}
Then the problem {\rm (\ref{quasi})} has a solution $u\in {\mathbb
W}^{2,1}_{p,q,(\mu)}(Q)$.
\medskip

{\bf 2}. Let the following assumptions be satisfied:
\begin{description}
\item[{(i)}] $1<p\leqslant q<\infty$,\ \ $\widehat{\mu}(p,q)>0$,\
\ $-1/p<\mu<\widehat{\mu}(p,q)$,\ \ $\partial\Omega \in {\cal
W}^2_{p,(\mu)}$; \item[{(ii)}] functions $a^{ij}$ and $a$ satisfy
the structure conditions {\rm (\ref{quasi1})}; \item[{(iii)}]
$b,\Phi\in\widetilde{\mathbb L}_{p,q,(\mu)}(Q)$; \item[{(iv)}]
$\Phi_1 \in \widetilde{\mathbb L}_{p_1,q_1,(\mu_1)}(Q)$,\quad
$p_1\leqslant q_1<\infty$,\ \
$\widehat{\mu}(p_1,q_1)>\max\{\mu_1,0\}$; \item[{(v)}]
$a(\cdot,z,\mathfrak p)$ is continuous w.r.t. $(z,\mathfrak p)$ in
the norm $\pmb{|\!|\!|}\cdot\pmb{|\!|\!|}_{p,q,(\mu),Q}$.
\end{description}
Then the problem {\rm (\ref{quasi})} has a solution $u\in
\widetilde{\mathbb W}^{2,1}_{p,q,(\mu)}(Q)$.
\end{sats}

\begin{proof}
The proof by the Leray--Schauder principle is also rather
standard, see, \cite[Ch.V, \S6]{LSU}. In the case when the leading
coefficients are continuous in $t$, these assertions were proved
in \cite[Theorem 4.3]{Na}. Corresponding a priori estimates in
\cite{Na}, see also \cite{LU1} and \cite{ApN}, do not require
continuity of $a^{ij}$ with respect to $t$, while the solvability
of the corresponding linear problem follows from Theorem
\ref{linear}.
\end{proof}

Note that in Theorem \ref{quasilinear} for $p>q$ we deal with ${\mathbb L}_{p,q,(\mu)}(Q)$ scale
while for $p<q$ we deal with $\widetilde{\mathbb L}_{p,q,(\mu)}(Q)$ scale. The reason is that all
the a priori estimates for quasilinear equations are based on the Aleksandrov--Krylov
maximum principle. Up to now this statement is proved only if the right-hand side
of the equation belongs to the space with stronger norm, see \cite{Na1}.

\section{Appendix. Estimates of some integral operators}\label{int-oper}

In this section we denote $x=(x',x'')$ where $x'\in{\mathbb
R}^{n-m}$, $x''\in{\mathbb R}^{m}$, $1\le m\le n$. Also we use the
notation
$$
R_{x}=\frac{|x''|}{|x''|+\sqrt{t-s}};\quad R_{y}=\frac{|y''|}
{|y''|+ \sqrt{t-s}}.
$$
The following two lemmas are generalizations of \cite[Lemmas 2.1 and 2.2]{Na},
where they are proved for $r=2$.
\begin{lem}\label{L_p}
Let $1< p<\infty$, and let the kernel ${\cal K}(x,y,t,s)$
satisfy for $t>s$ the inequality
\begin{equation}\label{kernel_k}
|{\cal K}(x,y;t,s)|\le C\,\frac {R_{x}^{\lambda_1+r}
R_{y}^{\lambda_2}}{(t-s)^{\frac {n+2-r}2}}\,
\frac{|x''|^{\mu-r}}{|y''|^{\mu}}\,\exp
\left(-\frac{\sigma|x-y|^2}{t-s} \right),
\end{equation}

\noindent where $\sigma>0$, $0<r\le 2$,
${\lambda_1}+{\lambda_2}>-m$,
\begin{equation}\label{mu_m}
-\frac mp-\lambda_1<\mu<m-\frac mp+\lambda_2.
\end{equation}
Then the integral operator ${\cal K}$, corresponding to the kernel
{\rm (\ref{kernel_k})}, is bounded in $L_p(\mathbb R^n\times\mathbb R)$.
\end{lem}

\begin{proof} By (\ref{mu_m}) there exist numbers  $\gamma_1$ and $\gamma_2$ such that
\begin{equation}\label{gamma}
-\frac mp<\gamma_1<\lambda_1+\mu, \qquad 0<\gamma_2<\frac
m{p'}+\lambda_2-\mu\, .
\end{equation}
Let $h\in L_p$. Applying (\ref{kernel_k}) and the H\"older
inequality, we have
\begin{multline}\label{est_l_p}
|({\cal K}h)(x,t)| \le C\,\biggl(\,\int\limits_{-\infty}^{t}\int\limits_ {{\mathbb
R}^n} \exp \left(-\frac{\sigma|x-y|^2}{t-s} \right)\frac
{|h(y,s)|^p\ R_{x}^{\gamma_1 p+r}\,R_{y}^{\gamma_2p}}
{|x''|^{(r-\mu) p}(t-s)^{\frac {n+2-r}2}}\ dyds\biggr)^{\frac{1}{p}}\\
\times\biggl(\,\int\limits_{-\infty}^{t}\int\limits_{{\mathbb R}^n}
\exp \left(-\frac{\sigma|x-y|^2}{t-s} \right)\frac {{\cal
R}_{x}^{(\lambda_1-\gamma_1)p'+r}{\cal
R}_{y}^{(\lambda_2-\gamma_2)p'}} {|y''|^{\mu p'}(t-s)^{\frac
{n+2-r}2}}\ dy ds \biggr)^{\frac{1}{p'}}.
\end{multline}
Let us denote by $I_3$ the last integral over
$(-\infty,t)\times\mathbb R^n$. Using the change of variable
$y=x-z\sqrt{t-s}$ in $I_3$ and, in the case $m<n$, integrating
there with respect to $z'$  after straightforward calculations  we
obtain
$$I_3=\int\limits_{0}^{t}\frac {R_{x}^{(\lambda_1-\gamma_1)p'+r}}
{(t-s)^{1-r/2}}\int\limits_{{\mathbb R}^m} \frac
{\exp(-\sigma|z''|^2)
|x''-z''\sqrt{t-s}|^{(\lambda_2-\gamma_2-\mu)p'}\ dz''}{\left(
|x''-z''\sqrt{t-s}|+\sqrt{t-s}\right)^{(\lambda_2-\gamma_2)p'}}\
ds.$$ By (\ref{gamma}) the  integral over $\mathbb R^m$ is
absolutely convergent and it is estimated by
$C(|x''|+\sqrt{t-s})^{-\mu p'}$. Therefore,
\begin{equation}\label{Mars3a}
I_3\le C\int\limits_{-\infty}^{t}\frac
{|x''|^{(\lambda_1-\gamma_1)p'+r}\ ds}
{(|x''|+\sqrt{t-s})^{(\lambda_1+\mu-\gamma_1)
p'+r}(t-s)^{1-r/2}}\leq C\,|x''|^{r-\mu p'}.
\end{equation}
We used here that the integral is absolutely convergent, since
$r>0$ and $\lambda_1+\mu-\gamma_1>0$ by (\ref{gamma}). Applying
this inequality for estimating the right-hand side in
(\ref{est_l_p}), we obtain
\begin{multline*}
\int\limits_{-\infty}^{\infty}\int\limits_{\mathbb R^n}|({\cal
K}h)(x,t)|^p\ dx dt\le
C\int\limits_{-\infty}^{\infty}\int\limits_{\mathbb R^n}|h(y,s)|^p
\ dy ds\\
\times
\sup\limits_{y,s}\int\limits_{s}^{\infty}\int\limits_{{\mathbb
R}^n} \exp \left(-\frac{\sigma|x-y|^2}{t-s} \right)\frac {{\cal
R}_{x}^{\gamma_1 p+r} R_{y}^{\gamma_2p}} {|x''|^{r}(t-s)^{\frac
{n+2-r}2}}\ dx dt.
\end{multline*}
Denote by $I_4$ the last integral over $(s,\infty)\times\mathbb R^n$.
Using the change of variable $y=x-z\sqrt{t-s}$ in $I_4$ and,
in the case $m<n$, integrating there with respect to $z'$,   we
obtain
$$
I_4=\int\limits_{s}^{\infty}\frac
{R_{y}^{\gamma_2p}}{(t-s)^{1-r/2}} \int\limits_{{\mathbb R}^m}
\frac {\exp(-\sigma|z''|^2) |y''-z''\sqrt{t-s}|^{\gamma_1 p}\
dz''} {\left(|y''-z''\sqrt{t-s}|+\sqrt{t-s}\right)^{\gamma_1
p+r}}\ dt.
$$
By (\ref{gamma}),  the  integral over $\mathbb R^m$ is absolutely
convergent and it is estimated by $C(|y'|+\sqrt{t-s})^{-r}$.
Therefore,
$$I_4\le C\int\limits_{s}^{\infty}\frac {|y''|^{\gamma_2 p}\ dt}
{(|y''|+\sqrt{t-s})^{\gamma_2 p+r}(t-s)^{1-r/2}}\leq C.$$ This
completes the proof.
\end{proof}

\begin{Rem}. {\rm Lemma \ref{L_p}} is also true  in the case $p=1$ or $p=\infty$.
The proof repeats with evident changes the proof presented above.
\end{Rem}

\begin{lem}\label{L_p_infty}
Under assumptions of Lemma \ref{L_p}, the operator ${\cal K}$ is bounded
in $\widetilde L_{p,\infty}(\mathbb R^n\times\mathbb R)$.
\end{lem}

\begin{proof}
Let $h\in \widetilde L_{p, \infty}$ and let $\gamma_1$ and
$\gamma_2$ satisfy (\ref{gamma}). Using (\ref{kernel_k}) and the
H\"older inequality, we have
\begin{eqnarray*}
|({\cal K}h)(x,t)|&\le& C\,\biggl(\int\limits_{-\infty}^{t}
\int\limits_{{\mathbb R}^n} \exp \left(-\frac{\sigma|x-y|^2}{t-s}
\right)\sup\limits_{s}
|h(y,s)|^p \\
&\times& \frac{ R_{x}^{\gamma_1 p+r}\ R_{y}^{\gamma_2 p}}
{|x''|^{(r-\mu) p}(t-s)^{\frac {n+2-r}2}} \ dyds
\biggr)^{\frac{1}{p}}\cdot I_3^{\frac{1}{p'}}\, ,
\end{eqnarray*}
where $I_3$ is the same as in the previous lemma. Applying
estimate (\ref{Mars3a}), we obtain
\begin{eqnarray*}
&&\int\limits_{{\mathbb R}^n} \sup\limits_t|({\cal K}h)(x,t)|^p\ dx
\le  C\, \int\limits_{{\mathbb R}^n} \sup\limits_{s}|h(y,s)|^p\
dy\\
&\times&
\sup\limits_{y}\int\limits_{0}^{\infty}\int\limits_{{\mathbb R}^n}
\exp \left(-\frac{\sigma|x-y|^2}{\tau} \right)\frac{|x''|^{\gamma_1
p} |y''|^{\gamma_2 p}\ dxd\tau}
{\left(|x''|+\sqrt{\tau}\right)^{\gamma_1p+r}
\left(|y''|+\sqrt{\tau}\right)^{\gamma_2p}\tau^{\frac {n+2-r}2}}.
\end{eqnarray*}
The last integral is estimated in the same way as $I_4$ from the
previous lemma. Therefore, it is bounded uniformly w.r.t. $y$, and the statement follows.
\end{proof}

The next lemma is a generalization of \cite[Lemma 3.2]{Na}.
\begin{lem}\label{L_p_1}
Let $1<p<\infty$, $\sigma>0$, $\varkappa>0$, $0\le r\le 2$,
${\lambda_1}+{\lambda_2}>-m$ and let $\mu$ be subject to {\rm
(\ref{mu_m})}. Also let
 the kernel ${\cal K}(x,y,t,s)$ satisfy
 the inequality
\begin{equation}\label{kernel_k1}
|{\cal K}(x,y,t,s)|  \le C\,\frac {R_{x}^{\lambda_1+r} {\cal
R}_{y}^{\lambda_2}}{(t-s)^{\frac {n+2-r}2}}
\,\frac{|x''|^{\mu-r}}{|y''|^{\mu}}\, \left(\frac
{\delta}{t-s}\right)^{\varkappa}\!\!\, \exp
\left(-\frac{\sigma|x-y|^2}{t-s} \right),
\end{equation}
for $t>s+\delta$. Then for any $s^0>0$ the norm of the operator
$${\cal K}\ :\ L_{p,1}(\mathbb R^n\times\ (s^0-\delta,s^0+\delta))\ \to \
L_{p,1}(\mathbb R^n\times\ (s^0+2\delta,\infty))$$ does not exceed
a constant $C$ independent of $\delta$ and $s^0$.
\end{lem}

\begin{proof}
Let $h\in L_{p, 1}$ be supported in the layer $|s-s^0|\le\delta$.
Using (\ref{kernel_k1}) and the H\"older inequality, we have
\begin{eqnarray}\label{1st}
|({\cal K}h)(x,t)|&\le& C\,\int\limits_{0}^{t}\frac
{\delta^{\varkappa} ds} {(t-s)^{\varkappa+1-r/2}}\nonumber \\
&\times&\biggl(\ \int\limits_{\mathbb R^n} \exp
\left(-\frac{\sigma|x-y|^2} {t- s}\right)\frac {|x''|^{(\mu-r)
p}R^{(\lambda_1+r)p}_{x}\ |h(y, s)|^p} {(t- s)^{\frac n2}}\
dy\biggr)^{\frac{1}{p}}\nonumber\\
&\times&\biggl(\ \int\limits_{\mathbb R^n}\exp
\left(-\frac{\sigma|x-y|^2}{t- s} \right) \frac {{\cal
R}_y^{\lambda_2p'}}{|y''|^{\mu p'}(t- s)^{\frac n2}}\ dy
\biggr)^{\frac{1}{p'}}.
\end{eqnarray}
Denote by $I_5$ the integral in the last large brackets. Using the
change of variable $y=x-z\sqrt{t- s}$ and, in the case $m<n$,
integrating with respect to $z'$, we obtain
$$I_5=C\int\limits_{\mathbb R^m}\frac {\exp \left(-\sigma |z''|^2 \right)
|x''-z''\sqrt{t- s}|^{(\lambda_2-\mu) p'}\ dz''}
{\left(|x''-z''\sqrt{t- s}|+\sqrt{t- s}\right)^{\lambda_2p'}}\leq
C \left(|x''|+\sqrt {t- s}\right)^{-\mu p'}\!.
$$
From this estimate and (\ref{1st}), it follows that
\begin{eqnarray*}
&&\int\limits_{s^0+2\delta}^{\infty}\Vert ({\cal K}h)(\cdot,
t)\Vert_p\ dt\le C\,\int\limits_{s^0+2\delta}^{\infty} \biggl(\
\int\limits_ {\mathbb R^n}\biggl(\ \int\limits_{-\infty}^t \biggl(\
\int\limits_{\mathbb R^n}
\exp \left(-\frac {\sigma|x-y|^2}{t- s} \right)\\
&\times&\ \frac {|x''|^{(\lambda_1+\mu)p}|h(y, s)|^p\ dy}
{{\left(|x''|+\sqrt {t- s}\right)^{(\lambda_1+\mu+r)p}}\ ({t-
s})^{\frac n2}}\biggr)^{\frac{1}{p}}\frac {\delta^{\varkappa}\ ds}
{(t- s)^{\varkappa+1-r/2}}\biggr)^p dx\biggr)^{\frac{1}{p}}dt
\end{eqnarray*}
Using Minkowski inequality, we estimate the right-hand side by
\begin{eqnarray*}
&&C\,\int\limits_{ s^0+2\delta}^{\infty}\int\limits_ { s^0-\delta}^
{ s^0+\delta}\ \frac {\delta^{\varkappa}\ ds dt}
{(t-s)^{\varkappa+1-r/2}}\,\biggl(\ \int\limits_{\mathbb R^n}
\int\limits_{\mathbb R^n}\exp \left(-\frac
{\sigma|x-y|^2}{t- s} \right)\\
&\times&\, \frac {|x''|^{(\lambda_1+\mu)p}|h(y, s)|^p\ dy dx}
{{\left(|x''|+\sqrt {t- s}\right)^{(\lambda_1+\mu+r)p}}\ ({t-
s})^{\frac n2}}\biggr)^{\frac{1}{p}}\\
&\le &C\,\int\limits_{s^0-\delta}^{ s^0+\delta} \| h(\cdot,s)\|_p\
ds \int\limits_ {s^0+2\delta}^{\infty}\frac {\delta^{\varkappa} dt}
{(t-s)^{\varkappa+1-r/2}} \cdot \sup\limits_yI_6^{\frac{1}{p}},
\end{eqnarray*}
where
$$
I_6=\int\limits_{\mathbb R^n} \exp \left(-\frac
{\sigma|x-y|^2}{t-s} \right) \frac {|x''|^{(\lambda_1+\mu)p}\ dx}
{{\left(|x''|+\sqrt {t-s}\right)^{(\lambda_1+\mu+r)p}}\ ({t-
s})^{\frac n2}}.
$$
In order to estimate $I_6$, we apply the change of variables
$x=z\sqrt{t- s}$ and $y=w\sqrt{t- s}$ and, in the case $m<n$,
integrate with respect to $z'$. This leads to
$$I_6=\frac C{(t- s)^{rp/2}}\,\int\limits_{\mathbb R^m}\frac
{\exp\left(-\sigma |z''-w''|^2 \right)\ |z''|^{(\lambda_1+\mu)p}\
dz''} {(|z''|+1)^{(\lambda_1+\mu+r)p}}\le C\,(t- s)^{-rp/2}.
$$
Thus,
$$\int\limits_{ s^0+2\delta}^{\infty}\Vert ({\cal K}h)(\cdot, t)\Vert_p\ dt
\le C\,\Vert h\Vert_{p,1}\, \sup\limits_{|s- s^0|<\delta}
\int\limits_ { s^0+2\delta}^{\infty}\frac {\delta^{\varkappa} dt}
{(t-s)^{1+\varkappa}}\le C\,\Vert h\Vert_{p,1},
$$
which completes the proof.
\end{proof}

\vspace{6mm}

\noindent {\bf Acknowledgements.} V.~K. was supported by the
Swedish Research Council (VR). A.~N. was supported by grants NSh.227.2008.1 and RFBR 09-01-00729.
He also acknowledges the Link\"oping University for the financial support of his visit in May 2008.

\end{document}